\documentclass[bj,numbers]{imsart}
\usepackage{amsmath,amsthm, amssymb, latexsym, mathtools}
\usepackage{enumitem}
\usepackage{verbatim}
\usepackage{bbm}
\usepackage{mathrsfs}
\usepackage{accents}

\usepackage{hyperref}

\startlocaldefs
\newcommand{\ubar}[1]{\underaccent{\bar}{#1}}
\theoremstyle{plain}
\newtheorem{thm}{Theorem}
\newtheorem{defin}[thm]{Definition}
\newtheorem{lem}[thm]{Lemma}
\newtheorem{cor}[thm]{Corollary}

\newtheorem{rem}[thm]{Remark}
\newtheorem{prop}[thm]{Proposition}

\newtheorem{cond}[thm]{Condition}
\newtheorem{ass}[thm]{Assumption}
\numberwithin{thm}{section}
\numberwithin{equation}{section}

\newcommand{\E}{\mathbb{E}}

\renewcommand{\H}{\mathbb{H}}
\newcommand{\R}{\mathbb{R}}

\newcommand{\W}{\mathbb{W}}

\newcommand{\N}{\mathbb{N}}

\newcommand{\mG}{\mathcal{G}}
\newcommand{\mF}{\mathcal{F}}
\newcommand{\mR}{\mathcal{R}}
\newcommand{\tv}{{\tilde{\vi}}}

\newcommand{\mL}{\mathcal{L}}
\newcommand{\mO}{\mathcal{O}}

\newcommand{\mW}{\mathcal{W}}

\newcommand{\mD}{\mathcal{D}}
\newcommand{\mM}{\mathcal{M}}

\newcommand{\mA}{\mathcal{A}}

\newcommand{\bk}{\tv}
\newcommand{\dr}{\mathcal D_R}

\newcommand{\eps}{\varepsilon}
\newcommand{\de}{\delta_\eps}

\newcommand{\der}{\frac{\delta_\eps^2}{\eps^2}}
\newcommand{\tde}{\tilde{\delta}_\eps}
\newcommand{\tder}{{\tde^2}/{\eps^2}}

\newcommand{\tsh}{\tilde{\mathcal{Q}}}
\newcommand{\tom}{\tilde{\mathcal{Z}}}
\newcommand{\toma}{\tilde{\mathcal{Z}}_\alpha}
\newcommand{\tomta}{\tilde{\mathcal{Z}}_{\tilde{\alpha}}}

\newcommand{\rd}{\R^d}
\newcommand{\ltwo}{L^2(\mO)}
\newcommand{\ltwod}{L^2(\mD)}

\newcommand{\cb}{C^\beta(\mO)}
\newcommand{\cad}{C^{\alpha-d-1}(\mO)}

\newcommand{\cbee}{C^b(\mO)}
\newcommand{\cbr}{C^\beta(\rd)}
\newcommand{\cbeer}{C^b(\rd)}
\newcommand{\ba}{B^\alpha_{11}(\mO)}
\newcommand{\bta}{B^{\tilde{\alpha}}_{11}(\mO)}

\newcommand{\bard}{B^\alpha_{11}(\rd)}

\newcommand{\bak}{B^{\alpha+\kappa}_{11}(\mO)}

\newcommand{\had}{H^{\alpha-\frac{d}2}(\mO)}
\newcommand{\hard}{H^{\alpha-\frac{d}2}(\rd)}
\newcommand{\hk}{(H^\kappa(\mO))^\ast}

\newcommand{\hkr}{(H^\kappa(\rd))^\ast}

\newcommand{\mAe}{\mA_\eps}
\newcommand{\vi}{\mathcal{V}}
\newcommand{\vast}{\mW_\ast}
\newcommand{\dast}{\mathcal{D}_\ast}
\newcommand{\jast}{J_\ast}
\newcommand{\psast}{\Psi_\ast}
\newcommand{\kmap}{\tilde K}

\newcommand{\tp}{\tilde{\Pi}}
\newcommand{\tf}{\tilde{F}}
\newcommand{\tc}{\tilde{S}}

\newcommand{\subcol}{\Psi_{\mO}}

\newcommand{\cona}{c_1}
\newcommand{\conb}{c_2}

\newcommand{\conf}{c_1}

\newcommand{\conh}{c_2}
\newcommand{\coni}{c}
\newcommand{\conj}{c_{1}}
\newcommand{\conk}{c_{2}}
\newcommand{\conl}{c_{3}}

\newcommand{\cons}{D}

\newcommand{\conu}{C''}

\newcommand{\crm}{c_M}

\newcommand{\norm}[1]{\|#1 \|}

\newcommand{\rp}{\mathbb{P}}
\newcommand{\ip}[2]{\big\langle#1,#2\big\rangle}
\newcommand{\supp}{{\rm supp}}
\newcommand{\pe}{\Pi_\eps}
\newcommand{\peg}{\Pi_\eps^{\mG}}

\newcommand{\del}{\eta} 
\newcommand{\K}{\mathcal{K}}
\newcommand{\tals}{\tau_\lambda^2} 
\newcommand{\jle}{\mathcal{J}_{\lambda,\eps}} 

\newcommand{\scaling}{\rho}
\newcommand{\dist}{\mu}
\newcommand{\Fpls}{\hat{F}_{PLS}}
\newcommand{\fpls}{\hat{f}_{PLS}}
\endlocaldefs

\begin{document}

\begin{frontmatter}


\title{Laplace priors and spatial inhomogeneity in Bayesian inverse problems}


\runtitle{Laplace priors in inverse problems}

\begin{aug}
\author[A]{\fnms{Sergios} \snm{Agapiou}\ead[label=e1]{agapiou.sergios@ucy.ac.cy}}
\author[B]{\fnms{Sven} \snm{Wang}\ead[label=e2]{svenwang@mit.edu}}
\address[A]{Department of Mathematics and Statistics,
University of Cyprus, Nicosia, Cyprus.
\printead{e1}}

\address[B]{Institute for Data, Systems and Society,
Massachusetts Institute of Technology, Cambridge, USA.
\printead{e2}}
\end{aug}
\runauthor{S. Agapiou and S. Wang}

\begin{abstract}
Spatially inhomogeneous functions, which may be smooth in some regions and rough in other regions, are modelled naturally in a Bayesian manner using so-called \textit{Besov priors} which are given by random wavelet expansions with Laplace-distributed coefficients. This paper studies theoretical guarantees for such prior measures -- specifically, we examine their frequentist posterior contraction rates in the setting of non-linear inverse problems with Gaussian white noise. Our results are first derived under a general local Lipschitz assumption on the forward map. We then verify the assumption for two non-linear inverse problems arising from elliptic partial differential equations, the \textit{Darcy flow} model from geophysics as well as a model for the \textit{Schr\"odinger equation} appearing in tomography. In the course of the proofs, we also obtain novel concentration inequalities for penalized least squares estimators with $\ell^1$ wavelet penalty, which have a natural interpretation as maximum a posteriori (MAP) estimators. The true parameter is assumed to belong to some spatially inhomogeneous Besov class $B^{\alpha}_{11}$, with $\alpha>0$ sufficiently large.
In a setting with direct observations, we complement these upper bounds with a lower bound on the rate of contraction for \textit{arbitrary} Gaussian priors. An immediate consequence of our results is that while Laplace priors can achieve minimax-optimal rates over $B^{\alpha}_{11}$-classes, Gaussian priors are limited to a (by a polynomial factor) slower contraction rate. This gives information-theoretical justification for the intuition that Laplace priors are more compatible with $\ell^1$ regularity structure in the underlying parameter.
\end{abstract}



\end{frontmatter}

\tableofcontents

\section{Introduction}

In many complex non-parametric statistical models, inference tasks are characterized by indirect and noisy measurement schemes of some unknown parameter $F$, in which small observation errors in the data may result in large reconstruction errors. In mathematical terms, this is encapsulated in the framework of statistical inverse problems, where there is a known and `ill-posed' forward map $\mathcal G$ between suitable function spaces, and noisy measurements of $\mathcal G(F)$ are available. Here, `ill-posedness' refers to the lack of a continuous inverse $\mathcal G^{-1}$. Prototypical examples arise from image reconstruction and tomography \cite{KS05}, geophysics \cite{ILS14, Y86}, seismology \cite{MWBG12} and an abundance of other areas of science and engineering. In particular, a large class of forward maps arise from partial differential equations (PDE) and are \emph{non-linear} \cite{I06}; our main examples below are of this type. In recent years, the Bayesian approach to inverse problems enjoyed enormous popularity since it offers a natural paradigm to quantify uncertainty via `credible sets' and since the computation of high-dimensional posterior distributions via scalable Markov Chain Monte Carlo has become more feasible than ever before \cite{S10, DS17, CRSW13} 
-- with Gaussian processes being arguably the most widely used prior models for $F$.

This paper studies the case of \textit{spatially inhomogeneous} functions $F$, i.e. functions which may be locally constant in some areas and exhibit high variation (or even jumps) in other areas. This is obviously relevant in imaging settings but also, for instance, in geophysics, when one aims to model physical properties of layered media. It is well-known that a natural mathematical representation for such functions is given by Besov spaces $B^{s}_{pq}$ with $s>0,~p=q=1$,
which measure the local change of $F$ in an $L^1$-sense. Minimax estimation problems for such unknowns have been studied widely in the literature, see, e.g., \cite{DJ98} for direct problems and \cite{DD95}, where linear inverse problems are considered. A central finding of these works is that (in $L^2$-loss) linear estimators can only achieve polynomially slower convergence rates than the minimax rate.

The key modelling choice in the Bayesian setting is a prior distribution which reflects appropriately the regularity structure present in the parameter, and the paper \cite{LSS09} first proposed a class of sequence priors with independently drawn and appropriately weighted \textit{Laplace-distributed} wavelet coefficients (see also \cite{LS04, DHS12, DS17, KLSS21}) in order to capture the $\ell^1$-nature of $B^s_{11}$-norms. In particular, this leads to a Bayesian analogue of placing an $\ell^1$ penalty on the wavelet coefficients termed $B^s_{11}$-\emph{Besov} priors (while Gaussian priors are associated to $\ell^2$-type Sobolev norm penalties via their RKHS). In this paper, we investigate two distinct notions of posterior consistency for the Bayesian recovery of the ground truth parameter using such Laplace priors in the `frequentist' small noise limit; the data are assumed to be corrupted by white noise. While our main goal is to study the rates of contraction for the full posterior distribution \cite{GGV00,GV17}, our proofs do require us to also examine the concentration properties of variationally defined penalized least squares estimators with $\ell^1$-type wavelet penalty, which can be shown to correspond to Maximum a Posteriori (MAP) estimators arising from Laplace wavelet priors \cite{ABDH18}. Finally, note that we are not attempting to reconstruct a \textit{sparse} signal for which posterior contraction rates with Laplace priors are known to be sub-optimal \cite{CSV15}.

Building on seminal work in the 2000s \cite{GGV00,GV07,VV08}, significant progress was made in the frequentist study of nonparametric Bayesian inverse problems, at first in the linear setting, see, e.g., \cite{KVV11, ALS13, KR13}. Only recently, similar results were derived for non-linear PDE models in the papers \cite{N20, NVW18, MNP19b, AN19, GN20, HK22} -- these references mostly consider Gaussian or `uniformly bounded' wavelet priors and Sobolev/H\"older-type regularity assumptions on the ground truth. On the other hand, the paper \cite{ADH21} initiated the study of posterior contraction rates for Laplace priors by examining the concentration properties of product measures of `$p$-exponential' type ($1\le p\le 2$), yielding contraction rates for standard `direct' nonparametric models (paralleling the results from \cite{VV08} for Gaussian process priors). Notably, those upper bounds \emph{suggest} that while Laplace priors can achieve the minimax rate of estimation over spatially inhomogeneous Besov bodies, Gaussian priors might be limited by the above-mentioned slower linear minimax rate. We also mention two recent preprints which respectively consider Laplace priors for stochastic diffusion models with Sobolev ground truth \cite{GR21} and spike-and-slab as well as random tree-type priors for `direct' nonparametric regression under spatially varying H\"older smoothness \cite{RR21}.


\subsubsection{Contributions}
A main contribution of our paper is to extend the contraction results for Laplace priors over inhomogeneous Besov bodies from \cite{ADH21} to (non-linear) inverse problems. We do this under a general local Lipschitz condition for the forward map $\mathcal G$, further relaxing regularity assumptions  imposed in previous (Gaussian and Sobolev) literature \cite{NVW18, N20}; see Theorem~\ref{thm:gen-contraction}  below as well as the discussion preceding it. Those conditions are satisfied for a range of representative non-linear forward maps, including for instance the \textit{Darcy flow} model from groundwater geophysics \cite{S10, Y86} where $\mathcal G(F)$ is given by the (unique) solution $u\equiv u_F$ to the boundary value problem 
\begin{equation}\label{div-intro}
    \nabla \cdot (\exp(F) \nabla u) = g ~\text{on}~\mO,~~~~~ u=0 ~\text{on}~\partial \mO. 
\end{equation}
Here, $\mathcal O\subseteq \R^d$ ($d\ge 1$) is a bounded and smooth domain, $\nabla \cdot$ denotes divergence and $g:\mathcal O\to (0,\infty)$ is a \textit{known}, smooth and positive source function, see Section~\ref{sec:pderes} for details. While we focus on (\ref{div-intro}) and one other representative example arising from an (elliptic) \textit{steady-state Schr\"odinger equation}, we expect that the conditions laid out in Section~\ref{sec:genres} are indeed satisfied by a far larger class of forward models.

We also address the following complementary question about lower bounds: \textit{Can Gaussian process priors (GPPs) achieve the same rates of contraction as Laplace priors over spatially inhomogeneous Besov function classes?} In other words, one asks whether the $\ell^1$-type penalty structure implicit in Laplace-type wavelet priors is more compatible, in an actual information theoretic sense, with $B^s_{pq} ~(1\le p<2)$ than commonly used Gaussian random fields. Theorem~\ref{thm-lb} below answers this question in the affirmative, in a setting with direct observations ($\mathcal G\equiv \textnormal{Id}$). Specifically, we find that any contraction rate for a sequence of GPP's, that is `uniform' over $B^s_{pq}$ bodies, is limited by the linear minimax rates from \cite{DJ98}. This confirms that the gap between the upper bounds with Gaussian and Laplace priors in \cite{ADH21} is indeed unavoidable, see Theorem 5.9 and Remark 5.10 there.



Our main results are stated in Sections~\ref{sec:genres} (upper bounds for general $\mathcal G$),~\ref{sec:pderes} (upper bounds for PDE models) and~\ref{sec-LB} (lower bounds). Section~\ref{sec:proofsfor2} contains the proofs for the upper bounds for general $\mathcal G$, while Section \ref{sec:prior} includes a summary of key measure-theoretic properties of Laplace priors. The remaining proofs, additional background material and some technical results on Besov spaces can be found in the Supplement.

\subsubsection{Proof ideas}

The structure of our proofs in Section~\ref{sec:proofsfor2} follows the general approach first laid out in \cite{GGV00} (see also the monographs \cite{GV17,GN16}) of demonstrating (a) that the prior satisfies exponential `prior mass around the true parameter' as well as `sieve set mass' inequalities, and (b) that sufficiently $L^2$-separated elements of the sieve set (equal to an appropriate enlargement of a Besov body here) can be tested against the ground truth, in a uniform manner. In our setting, we verify the former conditions using concentration results for log-concave product measures due to Borell \cite{CB74} and Talagrand \cite{T94}. To obtain the necessary tests, we employ novel concentration inequalities for $\ell^1$-type  penalised least squares estimators for $\mathcal G(F)$, adapting metric entropy-based techniques from $M$-estimation \cite{VDG00,V01} (similarly to \cite{NVW18}) to the present non-linear regression setting with locally Lipschitz $\mathcal G$ -- see Theorem~\ref{thm:plsrbddnorm} below. We also note that though the domains of forward operators $\mathcal G$ considered are generally unbounded, due to the `compactifying' effect of the prior distribution our results essentially only rely on \textit{local} forward and `backward stability' estimates on $\mathcal G$, allowing for an arbitrary scaling of the local Lipschitz constants for parameters far from the origin. In particular, this permits a more flexible choice of parameterisation (and of `link function'), in our main PDE examples, see Section~\ref{sec:pderes}. In contrast, most existing results do rely on an explicit quantitative (polynomial) control of relevant non-linearities, see, e.g., \cite{NVW18,GN20}.

Our lower bounds are proved by relating Gaussian process priors with the minimax estimation theory over Besov bodies \cite{DJ98} mentioned above. To do so in a rigorous manner, we require a \textit{uniform} notion of contraction rates, as formalized in (\ref{unif-contr}). While contraction rates in the literature are not generally formulated this way, the underlying proof techniques do usually permit such `stronger' uniform bounds, as demonstrated in Theorem~\ref{thm:gen-contraction} and its proof below. Our techniques stand in contrast to \cite{IC08} where lower bounds for contraction rates are derived for certain Gaussian priors by use of explicit properties of their `concentration function'; see Sections~\ref{subsec-contr} and~\ref{sec-LB} for more discussion.

\subsubsection{Outlook}
While Bayesian numerical computation can be expected to be challenging in the current PDE setting (due to the non-log-concavity of the posterior distribution), significant advances were made during the last decade in overcoming such computational barriers. For various formulations of MCMC methodologies in high- and infinite-dimensional settings we refer to the papers \cite{CRSW13,BGLFS17,CLM16} and \cite{CDPS18} where Laplace priors are considered, as well as references therein. There is also a growing literature in deriving rigorous theoretical `mixing' guarantees for those MCMC schemes, see, e.g., \cite{D17,DM19, HSV14, NW20, RS18}. The recent work \cite{NW20} derives polynomial-time convergence guarantees for high-dimensional sampling problems arising in the Schr\"odinger model studied in Section~\ref{sec-schr} below, when Gaussian priors are utilized. Conceivably, the approach taken in \cite{NW20} can be extended to obtain polynomial-time computation bounds for Laplace priors -- here, however, we leave this topic to be explored in future work. Similarly, the computation of MAP-type estimators can be challenging in non-linear inverse problems due to the non-convexity of the objective, and the development of rigorous computational guarantees is an ongoing research effort (see, e.g., \cite{NW20} as well as  \cite{EHN96, I06} for classical iterative methods).

Note that our convergence results below do require the true parameter to satisfy a minimum Besov-smoothness of at least $\alpha > d+1$. In particular, our results do not cover the `low-regularity' space $B^1_{11}$ associated to total variation \cite{DJ98, LSS09}. The main reasons are rooted in our use of classical elliptic PDE theory \cite{GT98} which requires H\"older-type smoothness of the coefficients, but also, secondarily, for standard metric entropy integrals for function spaces, employed to study the convergence of the direct problem, to converge. Nevertheless, the  spaces $B^\alpha_{11},~\alpha>1$, still capture meaningful spatially inhomogeneous behaviour, in that they penalize (higher-order) variations in an $L^1$-sense. A careful separate study of the low-regularity regime poses an interesting challenge for future research.

See also the Remarks~\ref{rem:scaling},~\ref{rem:minimax} and~\ref{rem:trunc} below for more discussion of possible future directions which relate to the use of `truncated' or `non-rescaled' priors, and minimax optimality.

\subsubsection{Basic notation and function spaces}
For any (nonempty) open, Borel subset $\mathcal O\subseteq \R^d$ ($d\ge 1$) with smooth boundary $\partial O$, let $L^2(\mathcal O)$ denote the (Lebesgue) square integrable functions on $\mathcal O$. We write $C(\mathcal O)$ for the space of bounded, continuous functions on $\mathcal O$, equipped with the sup-norm $\|\cdot\|_\infty$. For any integer $b\ge 1$, we write $C^b(\mathcal O)$ for the space of $b$ times continuously differentiable functions; $C_c^\infty(\mathcal O)$ denotes the smooth functions compactly supported in $\mathcal O$. For non-integer $b>0$, we denote by $C^b(\mathcal O)$ the usual $b$-regular spaces of H\"older-continuous functions.

For integer $s\ge 0$, let $H^s(\mathcal O)$ denote the space of $s$ times weakly differentable $L^2(\mathcal O)$ functions with derivatives in $L^2(\mathcal O)$, normed by
\[ \|F\|_{H^s(\mathcal O)}^2:=\sum_{|b|\le s} \|\partial^b F\|_{L^2(\mO)}^2, \]
and for non-integer $s>0$ we may define $H^s(\mO)$ by interpolation, see, e.g., \cite{LM72}. We moreover need the following subspace of $H^s(\rd)$-functions supported in $\bar{\mO}$,
\begin{equation}\label{eq:htilde} \tilde H^s(\mO) := \{ f\in H^s(\rd) : \supp (f)\subseteq \bar{\mO}\}.\end{equation}
For any Banach space $X$, we denote by $X^*$ its topological dual space $\{L:X\to \R~\text{linear and bounded} \}$, equipped with the operator norm. For $s>0$, we then define the negative order Sobolev spaces as $H^{-s}(\mO)=(\tilde H^{s}(\mO))^*$. For $s\ge 0$ and $p,q\in [1,\infty]$, we denote by $B^s_{pq}(\mO)$ the usual $s$-regular space of Besov functions on $\mathcal O$; we refer to Section B in the Supplement as well as \cite{ET08}, p.57, for more detailed definitions. For $p=q=\infty$, and non-integer $s>0$, we have $C^s(\mathcal O)=B^s_{\infty\infty}(\mathcal O)$, see, e.g., Section 4.3.4 in \cite{T83}. We will frequently use the (continuous) embedding $B^\alpha_{11}(\mathcal O)\subseteq C^b(\mathcal O)$ which holds for any $0<b<\alpha-d$; i.e. there exists some $C>0$ such that
\begin{equation}\label{embedding}
    \|f\|_{C^b(\mathcal O)}\leq C\|f\|_{B^\alpha_{11}(\mathcal O)}, ~~~\forall f\in B^\alpha_{11}(\mathcal O).
\end{equation}
[For non-integer $b$ this is a standard property of Besov spaces, and for integer $b$ we obtain this by choosing some non-integer $b'\in(b,\alpha-d)$ and observing that $B^\alpha_{11}(\mathcal O)\subseteq C^{b'}(\mathcal O)$.]

Finally, we use $\lesssim, \gtrsim$ and $\simeq$ respectively to denote one- and two-sided inequalities which hold up to multiplicative constants.




\section{Results with general forward maps in prediction risk}\label{sec:genres}

In this section, we first outline the construction of $B^\alpha_{11}$-Besov priors and other key preliminaries. We subsequently state a main result, Theorem~\ref{thm:gen-contraction}, about posterior contraction for $\mG(F)$ around $\mG(F_0)$. Section~\ref{sec:PLS} is devoted to convergence rates for penalized least squares estimators, see Theorem~\ref{thm:plsfr}.


\subsection{Statistical inverse problems and Besov-type prior distributions}\label{sec:fm}
Throughout this section, we fix integers $d,d'\ge 1$, an open (nonempty), bounded domain $\mO\subset \R^d$ with smooth boundary $\partial\mO$ and a Borel subset $\mathcal D\subseteq \R^{d'}$. For a measurable subset $\vi\subset \ltwo$, suppose that
\begin{equation}\label{eq:forwardmap}
\mG:\vi\to \ltwod
\end{equation}
is a given (measurable) `forward' map. For noise level $\eps>0$, a centered $L^2(\mathcal D)$-Gaussian white noise process $\W$ and any (unknown) $F\in\vi$, we assume that the data arise as a realisation of the stochastic process $(Y_\eps(\psi):\psi\in\ltwod)$ given by
\begin{equation}\label{eq:model}
Y_\eps=\mG(F)+\eps\W.
\end{equation}
Under mild regularity assumptions on the `regression functions' $\mathcal G(F)$, the observation scheme (\ref{eq:model}) constitutes a mathematically convenient, and asymptotically equivalent (in a Le Cam sense of statistical experiments) continuous limit for observing $\mathcal G(F)$ at order $\eps^{-2}$ many `equally spaced' points across $\mathcal D$, see, e.g., \cite{BL96,R08}. We note that $L^2(\mathcal D)$ and $\W$ here could generally be replaced by any separable Hilbert space $\H$ along with an `iso-normal' white noise process for $\H$ (akin to Section 2 in \cite{NVW18}), with essentially no changes in the proofs. We refrain from this generalization for ease of presentation. 



The law of $Y_\eps|F$ from (\ref{eq:model}) is denoted by $P^\eps_F$. By the Cameron-Martin theorem \cite[Proposition 6.1.5]{GN16}, the law $P^\eps_\W$ of $\eps\W$ is a common dominating measure for the laws arising from (\ref{eq:model}), and we may thus define the \emph{likelihood function}
\begin{equation}\label{eq:likelihood}
p^\eps_F(y)\coloneqq\frac{dP^\eps_F}{dP^\eps_\W}(y)=\exp\Big(\frac1{\eps^2}\ip{y}{\mG(F)}_{\ltwod}-\frac1{2\eps^2}\norm{\mG(F)}_{\ltwod}^2\Big).
\end{equation}
In the Bayesian approach one employs a prior (Borel) probability measure $\Pi$ supported on $\vi$ to model the unknown parameter $F$, and inference on $F$ is based on the \emph{posterior distribution} $\Pi(\cdot|Y_\eps)$ of $F|Y_\eps$. The latter is determined by Bayes' formula [see, e.g., p.7 in \cite{GV17}]
\begin{equation}\label{eq:post}
\Pi(B|Y_\eps)=\frac{\int_Bp^\eps_F(Y_\eps)d\Pi(F)}{\int_\vi p^\eps_F(Y_\eps)d\Pi(F)},~~~\textnormal{for any Borel set}~B\subseteq \vi.
\end{equation}


In this paper we consider \emph{Besov-type} prior distributions on $L^2(\mathcal O)$ which arise as random wavelet series expansions with Laplace-distributed weights. Specifically, consider an orthonormal basis $\Psi$ of $L^2(\R^d)$, consisting of `sufficiently smooth', compactly supported wavelets -- e.g., we may choose $\Psi$  to consist of Daubechies tensor wavelets of regularity $S>0$, see, e.g., \cite{D92,M92} or also \cite{N20}. [In what follows we will tacitly assume $S>\alpha$, where $\alpha$ denotes the smoothness parameter below.] We further denote by \begin{equation}\label{eq:subcol}\subcol= \{\psi\in \Psi:\supp(\psi)\cap \mathcal O \neq \emptyset \}= \{\psi_{kl}\}_{k\ge1, 1\le l\le L_k}\end{equation}
the sub-basis of wavelets whose support  has nonempty intersection with $\mO$, which we have enumerated at each level $k\ge 1$ by the index $1\le l \le L_k$. Since $\mathcal O$ is bounded and open, $L_k$ satisfies $L_k\simeq 2^{dk}$. Despite this sub-collection not necessarily forming an orthonormal system in $\ltwo$ (since some wavelets may have support with nonempty intersection with both $\mO$ and $\R^d\setminus\mO$), any square integrable function $F$ supported on $\mO$ can be uniquely represented as a sum
\begin{equation}\label{eq:wavexp}
F=\sum_{k=1}^\infty\sum_{l=1}^{L_k}F_{kl}\psi_{kl},~~~ \text{where}~ F_{kl}=\ip{F}{\psi_{kl}}_{\ltwo}.
\end{equation}

The following definition is similar to previous constructions of Besov-type priors \cite{LSS09, DHS12, ABDH18} on bounded intervals and the $d$-dimensional torus.

\begin{defin}[$B^\alpha_{11}$-Besov prior on $\ltwo$]\label{def:prior}
For integer regularity level $\alpha>d$ and i.i.d. univariate standard Laplace random variables $\{\xi_{kl}\}_{k\ge 1, 1\le l \le L_k}$, let
\begin{equation}\label{eq:Ftilde}
	\tf=\sum_{k=1}^\infty\sum_{l=1}^{L_k} 2^{(\frac{d}2-\alpha)k} \xi_{kl} \psi_{kl}.
\end{equation}
We denote the law of $\tf$ by $\tp$. Next, given any compact subset $K\subset \mO$ (which will be chosen below), fix some smooth cut-off function $\chi\in C_c^\infty(\mO)$ satisfying $\chi=1$ on $K$. Then, for some scalar $\scaling>0$ (also to be chosen later), the $B^\alpha_{11}$-Besov prior distribution is given by
\begin{equation}\label{prior}
\Pi=\mL(\scaling \chi \tf), ~~ \tf\sim \tp.
\end{equation}
\end{defin}
The compact subset $K\subset \mathcal O$ and cut-off function $\chi$ are employed to deal with well-known intricacies with wavelet representations at the boundary $\partial O$ (cf. \cite{GN20,T08}). In addition to the regularity level $\alpha$, here $\rho>0$ is a scaling parameter which allows to calibrate the spread of the prior distribution. Note that the assumption $\alpha>d$ ensures the summability of the sequence $(L_k 2^{(d-2\alpha)k}:k\ge 1)$ such that $\Pi$ is supported on $\ltwo$. Under the stricter assumption $\alpha>1+d$, it follows from Lemma \ref{lem:supp} below that
\begin{equation}\label{eq:Cbetasupp}
\Pi(\cbee)=1,
\end{equation}
for any integer $b$ such that $1\leq b<\alpha-d$, and for any $\rho>0$. In fact, below we choose $\rho$ as an appropriate function of the noise level $\eps$, which secures that the prior concentrates on a fixed ball of $\cbee$; see the proof of Lemma~\ref{lem:forwardpriormass}.

\subsection{Posterior contraction rates for locally Lipschitz forward maps}\label{subsec-contr}

Our first main result, Theorem~\ref{thm:gen-contraction} below, regards contraction properties for the forward-level posterior distribution on $\mathcal G(F)$ around $\mathcal G(F_0)$. Its main hypotheses on the prior measure, ground truth and forward map~$\mG$ are summarized in Assumptions~\ref{ass:truth} and~\ref{ass:kgb}.

\begin{ass}[Prior and ground truth]\label{ass:truth} We assume the following.
\begin{enumerate}[label=\textbf{\roman*)}]
\item Let $\alpha > 1+d$ integer and let $K\subseteq \mathcal O$ be a compact subset. Suppose that data is given by $Y_\eps$ from (\ref{eq:model}) with some ground truth parameter $F_0\in \ba\cap \vi$ with $\supp(F_0)\subseteq K$.


\item For those choices of $\alpha, K$, some cut-off function $\chi\in C^\infty_c (\mathcal O)$ with $\chi\equiv 1$ on $K$, and some integer $\kappa \ge 0$, let $\Pi_\eps$ be the $B^\alpha_{11}$-Besov prior from Definition~\ref{def:prior}, with scaling constant
\begin{equation}\label{eq:deltaeps}
\scaling = \eps^2/\de^2 = \eps^{\frac{2d}{2\kappa+2\alpha+d}}, ~~\text{where}~\delta_\eps=\eps^\frac{2\kappa+2\alpha}{2\kappa+2\alpha+d}.
\end{equation}
\end{enumerate}

\end{ass}
The above assumption, together with \eqref{embedding}, implies that $F_0\in \cbee$, for $1\le b<\alpha-d$. The next assumption entails that the forward map $\mathcal G$ is locally Lipschitz with respect to a negative order Sobolev norm. An integer parameter $\kappa\ge 0$ in (\ref{eq:local-lip}) encapsulates the `degrees of smoothing' of the forward map (and $\kappa$ in the above assumption will then be chosen as that constant).



\begin{ass}[Forward regularity]\label{ass:kgb} For $\chi$ and $\alpha$ as in Assumption~\ref{ass:truth} and  some integer $1\leq \beta<\alpha-d$, suppose the domain $\vi$ of the forward map satisfies
	\begin{equation}\label{domain-lb-req}
		\vi\supseteq C^\beta_\chi(\mathcal O) := \big\{ F\in C^\beta(\mathcal O) : \supp(F)\subseteq \supp(\chi) \big\}.
	\end{equation}
	Moreover, assume that for some integer $\kappa\ge 0$ and any $R>0$ there exists $C_R>0$ such that
	\begin{equation}\label{eq:local-lip}
		\norm{\mG(F_1)-\mG(F_2)}_{\ltwod}\le C_R\norm{F_1-F_2}_{\hk},
	\end{equation}
    for any $F_1, F_2\in C^\beta_\chi(\mathcal O)$ with $\|F_1\|_{\cb}\vee\|F_2\|_{\cb}\le R$.
\end{ass}

The requirement (\ref{domain-lb-req}) will ensure that the local Lipschitz estimate (\ref{eq:local-lip}) is satisfied on a sufficiently large `sieve set' carrying the bulk of the prior (and posterior) mass; the estimate (\ref{eq:local-lip}) in turn will guarantee that the statistical complexity (i.e. metric entropy) of the set of induced regression functions $\{\mathcal G(F)\}_F$ can be suitably controlled. The assumption (\ref{eq:local-lip}) generalizes similar (but more restrictive) local Lipschitz assumptions requiring a \textit{polynomial} dependence of $C_R$ on $R$ which have been utilized in \cite{NVW18} to prove convergence of penalized-least-squares estimators with Sobolev penalties and in \cite{GN20} for studying rates of contraction under Gaussian priors.




\begin{thm}\label{thm:gen-contraction}
	Suppose that Assumptions~\ref{ass:truth} and~\ref{ass:kgb} are fulfilled. Then there exist constants $L,C_1$ and $\mu>0$ such that for all $\eps>0$ small enough,
	\begin{equation}\label{eq:forwardcontraction}
	 P_{F_0}^\eps\big(\pe\big( F: \|\mG(F)- \mG(F_0)\|_{\ltwod} \ge L\delta_\eps \,\big|\, Y_\eps \big)\ge e^{-\delta_\eps^2/\eps^2} \big)\le C_1\eps^{\mu}.
	\end{equation}
	Furthermore, for any $1\leq b<\alpha-d$ and sufficiently large $M,C_2>0$, we have that
	\begin{equation}\label{eq:Cbconcentration}
	P_{F_0}^\eps\big(\pe\big(F:\norm{F}_{\cbee}>M\,\big|\, Y_\eps \big)\ge e^{-\delta_\eps^2/\eps^2} \big)\le C_2\eps^{\mu}.
	\end{equation}
	The above constants only depend on $F_0$ through $\|F_0\|_{B^\alpha_{11}}$; in particular they can be chosen uniformly over $F_0\in \vi \cap \{F\in B^\alpha_{11}: \|F\|_{B^{\alpha}_{11}}\le r,~\textnormal{supp}(F)\subseteq K \}$, $r>0$.
\end{thm}

The rate of convergence $\de$ corresponds to the minimax estimation rate for (directly observed) $\alpha+\kappa$-smooth functions (in particular, for functions in $\bak$) under $L^2$-loss, see, e.g., \cite{DJ98} or \cite{DD95}. This is in accordance with Assumption~\ref{ass:kgb} which requires that $\mG$ is locally `$\kappa$-smoothing', while permitting arbitrary growth of the local Lipschitz constant as a function of the radius of the bounded set.

The contraction rate statement (\ref{eq:forwardcontraction}) differs from the existing literature \cite{GGV00, GN16,GV17} in that it aims to quantify `non-asymptotically' all terms that depend on the noise-level. As an immediate consequence of Theorem~\ref{thm:gen-contraction}, we also obtain the more `classical' asymptotic statements
\begin{equation}\label{eq:standard}\pe\big( F: \|\mG(F)- \mG(F_0)\|_{\ltwod} \ge L\delta_\eps \,\big|\, Y_\eps \big) = O_{P_{F_0}^\eps}\big( e^{-\delta_\eps^2/\eps^2} \big) = o_{P_{F_0}^\eps}(1). \end{equation}
However, the stronger assertion from (\ref{eq:forwardcontraction}) is granted by the same proof techniques as those typically used for establishing \eqref{eq:standard}; see Section~\ref{sec:24proof} for details. The quantitative bound allows us also to assert uniformity over Besov balls, which is crucial for controlling the worst-case risk in Theorem~\ref{thm-lb} in order to rigorously assert that Laplace priors outperform Gaussian priors over the spatially inhomogeneous $B^\alpha_{11}$ classes.

\begin{rem}[Smoothness and scaling of prior]\normalfont\label{rem:scaling}
    Draws from the prior distribution chosen in Theorem~\ref{thm:gen-contraction} almost surely do not belong to the space in which the truth is assumed to live, $B^\alpha_{11}$, but only to $B^{b}_{11},~ b<\alpha-d$, see Lemma \ref{lem:supp} below. In this sense, the prior is `undersmoothing' which heuristically is counteracted by the rescaling $\rho\to0$ in (\ref{eq:deltaeps}). The space $B^\alpha_{11}$ is closely related to the prior 
    via its concentration and absolute continuity properties, cf. Lemma \ref{lem:shsb} (and surrounding discussion) below, and in this sense generalizes the Reproducing Kernel Hilbert Space (RKHS) of Gaussian priors, cf. \cite{VV08, ADH21}. This undersmoothing property of the prior contrasts earlier results for Gaussian process priors in nonparametric regression and linear inverse problems under $\alpha$-Sobolev smoothness of the truth, where minimax rates could be obtained for Gaussian priors with RKHS-norm $\|\cdot\|_{H^{\alpha+d/2}}$ (that is '$\alpha$-smooth' priors, whose draws are almost surely in $H^{\alpha-\eta}$ for any $\eta>0$) without rescaling, cf. \cite{VV08, KVV11}.
    
    Both in the context of non-linear inverse problems, as well as for Laplace-type priors in direct models under $B^\alpha_{11}$-regularity of the truth, the necessity of choosing such undersmoothing and rescaled prior distributions has been observed; see, e.g. \cite{ADH21, GN20, NVW18}. Whether this is an artifact of the proof techniques at hand, is an interesting open question for future research.
\end{rem}

\subsection{Convergence rates for penalized least squares estimators}\label{sec:PLS}
We now turn to studying the concentration properties of penalized-least-squares (PLS) estimators with $\ell_1$-type Besov-norm penalty. These estimators arise naturally as Maximum a Posteriori (MAP) estimates from Besov priors \cite{ABDH18}. While the results of this section are of interest independently of the contraction rates studied above, their proof techniques are crucially used to obtain the statistical testing bounds required for posterior contraction rates, see the proof of Theorem~\ref{thm:gen-contraction} in Section~\ref{sec:24proof} and Theorem~\ref{thm:plsrbddnorm} in Section~\ref{sec:proof26}.

Though the regularity assumptions in this section will slightly differ from those above, we again assume that some forward map $\mathcal G:\vi\to L^2(\mathcal D)$ is given, where $\vi\subseteq L^2(\mathcal O)$ denotes its domain. For an arbitrary and fixed compact subset $\kmap\subseteq \mO$, we denote
\begin{equation}\label{eq:vi}
\tv\coloneqq \{F\in B^\alpha_{11}(\mO):\supp(F)\subseteq \kmap \},
\end{equation}
and we assume that $\tv\subseteq\vi.$ Since the space $\tv$ consists of functions which are compactly supported in $\mathcal O$, it follows from standard characterisations of Besov spaces that on $\tv$, the ${\ba}$-norm is equivalently characterized by the wavelet sequence norm 
\begin{equation}\label{eq:tomanorm}
\norm{h}_{\toma}:=\sum_{k=1}^\infty2^{(\alpha-\frac{d}2)k}\sum_{l=1}^{L_k} |h_{kl}|,~~~ h\in \tv,
\end{equation}
where $h_{kl}=\langle h,\psi_{kl}\rangle_{L^2}$ denotes the wavelet coefficients arising from the collection of wavelets $\Psi_{\mathcal O}$ from Section~\ref{sec:fm}. In particular, $\tv\subseteq \{h\in L^2(\mathcal O): \|h\|_{\toma}<\infty \}$. [See also display \eqref{eq:tom} for a definition of the sequence space $\tom=\toma$ which is intimately connected to the concentration properties of the $B^\alpha_{11}$-Besov prior from Definition~\ref{def:prior}.]

For data $Y_\eps$ arising from the white noise regression model (\ref{eq:model}) and recalling its likelihood function (\ref{eq:likelihood}), we consider estimators $\Fpls$ arising as maximisers of the penalized likelihood objective
\begin{equation}\label{Jdef}
\jle:\tv\to\R,\quad \jle(F)\coloneqq 2\ip{Y_\eps}{\mG(F)}_{\ltwod}-\norm{\mG(F)}^2_{\ltwod}-\lambda\norm{F}_{\toma}.
\end{equation}
Here $\lambda>0$ is a regularization parameter to be chosen. 


Akin to \cite{NVW18, GN20} we make the regularity assumption on the forward map $\mG$, that there exists $C>0$ and integers $\kappa, \gamma, \beta\ge 0$ such that for all $F_1, F_2\in \tv \cap \cb$,
\begin{equation}\label{eq:kgb}
	\norm{\mG(F_1)-\mG(F_2)}_{\ltwod}\le C \big(1+\norm{F_1}^\gamma_{\cb}\vee\norm{F_2}^\gamma_{\cb} \big)\norm{F_1-F_2}_{\hk}.
\end{equation}
When (\ref{eq:kgb}) is satisfied we say that $\mG$ is \emph{$(\kappa, \gamma, \beta)$-regular}. Note that (\ref{eq:kgb}) requires a more quantitative control over the local Lipschitz constant of $\mathcal G$ than (\ref{eq:local-lip}). For any $\lambda>0$, $F_1\in \tv$ and $F_2\in \vi$, we define the functional
\begin{equation}\label{eq:tals}
\tals(F_1,F_2)\coloneqq \norm{\mG(F_1)-\mG(F_2)}_{\ltwod}^2+\lambda\norm{F_1}_{\toma},
\end{equation}
whose concentration properties are the subject of the following main theorem.

\begin{thm}\label{thm:plsfr}
Let $\alpha, \kappa, \gamma, \beta \ge 0$ be integers such that \[\alpha>\beta+d,~~\alpha \ge d/2+d\gamma-\kappa,\]
and such that the forward map $\mG:\vi\to \ltwod$ is $(\kappa,\gamma,\beta)$-regular in the sense of (\ref{eq:kgb}). Suppose that data arise as $Y_\eps\sim P^\eps_{F_\dagger}$ for some fixed $F_\dagger\in\vi$ and let $\de$ be given by \eqref{eq:deltaeps}.
\begin{enumerate}[label=\textbf{\roman*)}]
\item For all $\lambda,\eps>0$, almost surely under $P_{F_\dagger}^\eps$, there exists a maximizer $\Fpls$ of $\mathcal J= \jle$ over $\tv$, satisfying
\[ \mathcal J(\Fpls)= \sup_{F\in \tv} \mathcal J(F). \]
\item There exist large enough constants $C,C',M>0$ (independent of $F_\dagger\in\vi$) such that with
\[\lambda=C\de^2,\]
we have that for any $0<\eps<1$, $L\ge M$, $F_\ast\in\tv$ and any $\Fpls\in \arg\max_{F\in \tv}\mathcal J (F)$,
\begin{equation}\label{eq:concbd}
P^\eps_{F_\dagger}\big(\tals(\Fpls,F_\dagger)\geq 2(\tals(F_\ast, F_\dagger)+L^2\de^2)\big)\leq C'\exp\Big(-\frac{L^2\delta_\eps^2}{C'\eps^2}\Big).
\end{equation}
In particular, there exists a universal constant $C''>0$ such that 
\[E^\eps_{F_\dagger}[\tals(\Fpls, F_\dagger)]\leq C''(\tals(F_\ast, F_\dagger)+\de^2).\]
\end{enumerate}
\end{thm}

Note that in the above theorem, we do not necessarily require the data-generating `true' parameter $F_\dagger$ to lie in the $\alpha$-regular set $\tv$, but merely that it can be approximated sufficiently well (in the sense of the functional $\tals$) by some $F_*\in \tv, ~F_*\approx F_\dagger$. [In fact this will be used crucially in the construction of the tests necessary for establishing contraction rates in the proof of Theorem \ref{thm:gen-contraction}]. Of course, when $F_\dagger\in\tv$ one may choose $F_\ast=F_\dagger$. We moreover observe that clearly the functional $\tau_\lambda^2(\Fpls, F_\dagger)$  only controls the \textit{prediction risk} $\|\mathcal G(\Fpls)-\mathcal G(F_\dagger)\|$, while convergence rates for  $\Fpls-F_\dagger$ typically additionally require a stability estimate for the inverse $\mathcal G^{-1}$; see Section~\ref{sec:pderes} below for concrete examples. We also refer to Theorem 2 in \cite{NVW18}, where analogous convergence results are derived for regularized least squares functionals with $L^2$-type Sobolev penalties.

An immediate consequence of the preceding theorem is the following corollary bounding the (more commonly considered) $L^2(\mathcal D)$-mean squared error.

\begin{cor}\label{cor-mse}
Suppose that $\alpha,\kappa,\gamma,\beta$ and $\mathcal G$ are as in Theorem~\ref{thm:plsfr}, and let $\mathcal J_{\lambda,\eps}$ and $\delta_\eps$ be given by (\ref{Jdef}) and (\ref{eq:deltaeps}) respectively. Then, for any $r>0$ there exist $C,C'>0$ such that with $\lambda= C\delta_\eps^2$ and for all $\eps>0$ small enough,
\begin{equation}\label{eq:expbd}
\sup_{F_\dagger\in\tv: \norm{F_\dagger}_{\ba}\leq r}E^\eps_{F_\dagger}\big\|\mG(\Fpls)-\mG(F_\dagger)\big\|_{\ltwod}^2\leq C' \delta_\eps^2.
\end{equation}
\end{cor}

To conclude this section, we remark that our proof techniques permit more flexible choices of penalty norms in principle, too.

\begin{rem}[Alternative choices of $\lambda$]
The preceding results make a choice $\lambda\simeq \delta_\eps^2$ which both yields the minimax-optimal convergence rate $\delta_\eps=\eps^{\frac{2\kappa+2\alpha}{2\kappa+2\alpha+d}}$ for estimating an $\alpha+\kappa$-smooth function, and is also tailored towards our proofs for Theorem~\ref{thm:gen-contraction}. Our proof techniques do in principle permit a more flexible choice of $\lambda$, which would lead to a result under assumptions on the relationship between $\lambda, \epsilon$ and the convergence rate $\delta$, paralleling (2.7) in \cite{NVW18}. 
\end{rem}


\section{Results for PDE inverse problems}\label{sec:pderes}

We now apply our general theory to two prototypical non-linear inverse problems arising from elliptic partial differential equations, one with a divergence form \textit{steady-state diffusion equation} and one with the \textit{steady-state Schr\"odinger equation}.

\subsection{Preliminaries}
Throughout this section we fix an open and bounded domain $\mathcal O\subseteq \R^d$ ($d\ge 2$) with smooth boundary $\partial O$ on which the PDEs are considered. In both of our examples the unknown parameter, denoted by $f:\mathcal O\to (0,\infty)$, is a \textit{positive} coefficient function of the differential operator at hand, and the forward map arises as the parameter-to-solution map $f\mapsto u_f\in L^2(\mathcal O)$ of a corresponding boundary value problem which we define in detail below. The observations are then given as a noisy version of the solution $u_f$, corrupted by a Gaussian white noise process $\W$ indexed by the Hilbert space $L^2(\mathcal O)$,
\begin{equation}\label{eq:modelf}
Y_\eps=u_f+\eps\W,~~ \eps>0.
\end{equation}
We denote the law of $Y_\eps$ by $P_{f}^\eps$, and recall that (\ref{eq:modelf}) serves as a continuous limit for `equally spaced' discrete measurements of $u_f$ \cite{R08,BL96}. 


Our interest is to model $f$ as a sufficiently smooth member of an $\ell^1$-type Besov space; as in Section~\ref{sec:genres} we fix some compact subset $K\subset \mathcal O$. Then, for regularity level $\alpha>d$ and lower bound $K_{min}\in [0,1)$, define the parameter spaces
\begin{equation}\label{Fdef}
\mathcal F=\mathcal F_{\alpha, K_{min}} = \big\{f\in B^\alpha_{11}(\mO): f\equiv 1~ \text{on}~\mO\setminus K, ~ f>K_{min} ~\text{on}~\mO\big\},
\end{equation}
to which the ground truth $f_0$ will be assumed to belong. To employ the Laplace-type priors from Definition~\ref{def:prior} to model $f\in\mathcal F$, which naturally are supported on \emph{linear} subspaces of $L^2(\mathcal O) $ and thus do \emph{not} obey positivity constraints, we need to introduce a routine one-to-one re-parameterisation $f\mapsto F:=\Phi^{-1}\circ f$, via some \emph{link function} $\Phi:\R\to (K_{min},\infty)$. Specifically, we will assume that $\Phi$ is an arbitrary (but fixed) smooth, one-to-one function satisfying $\Phi'(x)>0,~ x\in \R$; for convenience we will moreover assume that $\Phi(0)=1$. For instance, we may choose $\Phi(x)=(1-K_{min})\exp(x)+K_{min}$.


With the preceding definitions, the class of priors $\Pi^f$ which is considered for $f\in\mathcal F$ are simply given as pushforward distributions under $\Phi$ of the Laplace-type priors from Section~\ref{sec:genres}. Specifically, as before we fix a cut-off function $\chi\in C_c^\infty(\mathcal O)$ with $\chi|_{K}\equiv 1$. Then, for (integer) $\alpha >d$, scaling parameter $\rho>0$, and with $F\sim \Pi$ from (\ref{prior}), we set
\begin{equation}\label{f-prior}
    \Pi^f\coloneqq \mathcal L(\Phi\circ F) =\mathcal L\Big(\Phi\circ \Big[\rho\chi \cdot \sum_{k=1}^\infty \sum_{l=1}^{L_k}2^{(d/2-\alpha)k}\xi_{kl}\psi_{kl} \Big]\Big),
\end{equation}
where again $\{\xi_{kl}\}_{k,l}$ are i.i.d. Laplace random variables. Finally, we note that the choices of $f\equiv 1$ on $\partial O$ and $K_{min}\in [0,1)$ in (\ref{Fdef})  were made for convenience, and that the results to follow could be extended to more general boundary assumptions. Similarly, though it could be relaxed, the assumption for $\partial \mathcal O$ to be smooth is convenient to apply standard PDE regularity theory, see, e.g. \cite{LM72,T83, GT98}.

\subsection{Steady-state diffusion equation}
For a \textit{given} and known `source function' $g\in C^\infty(\mO)$ and `diffusion coefficient' $f:\mathcal O\to (0,\infty)$,
consider the divergence form PDE describing $u\equiv u_f$,
\begin{equation}\label{div}
\begin{cases}
	\nabla\cdot (f\nabla u) &= g ~~~\text{on}~ \mO, \\ 
	u &=0~~~ \text{on}~ \partial\mO. 
\end{cases}
\end{equation}
Whenever $f\in C^{1+\eta}(\mathcal O)$ for some $\eta>0$, it is well-known from classical (Schauder) PDE theory that a unique solution $u_f\in C^{2+\eta}(\mO)\cap C(\bar{\mO})$ to (\ref{div}) exists, see e.g. Chapter 6 of \cite{GT98}. By the standard embedding (\ref{embedding}) of Besov spaces, this can e.g. be ensured if $f\in\mF$ from (\ref{Fdef}), for  $\alpha> 1+d$. Note that while the PDE is linear, the map $f\mapsto u_f$ is nonlinear. The model (\ref{div}) can mathematically be viewed as a steady-state diffusion (or heat) equation \cite{R81} with unknown diffusivity (or conductivity) $f>0$. This PDE appears in a variety of application areas, including the modelling of subsurface geophysics and groundwater flow, see, e.g., \cite{S10,Y86}. In the mathematical literature, (\ref{div}) has served as a prototypical example of non-linear inverse problems \cite{S10,R81,DS17,itokunisch,BCDPW17}.

The general convergence results laid out in Section~\ref{sec:genres} permit us to derive both contraction rates of posterior distributions and convergence rates for variationally defined MAP-type PLS estimators, for the forward and the inverse problem. The proofs for this section can be found in Section A of the Supplement.

\begin{thm}[Posterior contraction]\label{div-fwdrate}
	For $K_{min}\in(0,1)$ and integer $\alpha > 2+d$, let $\mathcal F$, $u_f$ be given by (\ref{Fdef}), (\ref{div}) respectively, and suppose that $f_0\in \mF$. Moreover, define $\de:=\eps^{\frac{2\alpha+2}{2\alpha+2+d}}$
	and let the prior distribution $\pe^f\equiv \Pi^f$ for $f$ be given by (\ref{f-prior}) with $\rho=\eps^2/\de^2=\eps^{2d/(2\alpha+2+d)}$. Then the following holds.
	
	
	\textbf{i)} There exists a large enough constant $L>0$ and some $c>0$ such that as $\eps\to0$,
	\begin{equation*}
	\pe^f \big( f: \|u_f - u_{f_0}\|_{\ltwo} \ge L\de \,\big|\, Y_\eps\big) = O_{P_{f_0}^\eps}\big( e^{-c\de^2/\eps^2} \big).
	\end{equation*}
	
	
	\textbf{ii)} If in addition $g\ge g_{min}$ for some $g_{min}>0$, then it holds that for some $L, c>0$, as $\eps\to0$,
		\begin{equation*}
		\pe^f \big( f: \|f - f_0\|_{\ltwo} \ge L\de^{\theta} \,\big|\, Y_\eps\big) = O_{P_{f_0}^\eps}\big( e^{-c\de^2/\eps^2} \big),~~~\textnormal{where}~\theta = \frac{\alpha-d-2}{\alpha-d}.
		\end{equation*}
    For any $r>0$, these constants can be chosen uniformly over classes $\mathcal F^{(r)}:=\{f\in \mathcal F :\|\Phi^{-1}\circ f\|_{B^\alpha_{11}(\mathcal O)}\le r\}$.
\end{thm}

The proof idea is as follows. As in \cite{NVW18}, we use forward estimates for the non-linear PDE solution map $f\mapsto u_f$ to prove that the operator $F\mapsto u_{\Phi\circ F}$ indeed satisfies the local Lipschitz Assumption~\ref{ass:kgb} with $\kappa=1$, $\mathcal D=\mO$ (and suitable $F$), whence an application of Theorem~\ref{thm:gen-contraction} yields the prediction risk bound in part \textbf{i)} of the theorem. Part \textbf{ii)} of the Theorem additionally requires a `stability estimate' bounding $f-f_0$ in terms of $u_f-u_{f_0}$, in suitable norms. In non-linear inverse problems this requires a case-by-case analysis, and for the divergence form problem (\ref{div}) at hand this has been extensively studied in the last decades, see for instance \cite{R81,itokunisch,BCDPW17} and references therein. A quantitative version of such stability estimates was proved in \cite{NVW18}, which will be used in our proofs, in combination with standard interpolation inequalities for Besov spaces.



We now turn to the penalized least squares estimators with $\ell^1$-type Besov penalty from Section~\ref{sec:PLS}. We recall the wavelet penalty norm $\|\cdot\|_{\tilde{\mathcal Z}_\alpha}$ from (\ref{eq:tomanorm}) and its natural association to product Laplace measures via Lemma~ \ref{lem:shsb}. Again, to accommodate the (non-linear) positivity constraint from (\ref{Fdef}), the re-framed optimization objective from (\ref{Jdef}) in terms of $f$ takes the form
\begin{equation}\label{I-def}
    \mathcal I_{\lambda,\eps}:\mF\to\R,\quad \mathcal I_{\lambda,\eps}(f)\coloneqq 2\ip{Y_\eps}{u_{f}}_{\ltwod}-\norm{u_{f}}^2_{\ltwod}-\lambda\norm{\Phi^{-1}\circ f}_{\toma}.
\end{equation}
For the next result, we require an additional regularity assumption on the link function $\Phi$. As in \cite{NVW18} (see also \cite{GN20}), we make the following definition.

\begin{defin}\label{def:reglink}
	A map $\Phi:\R\to (K_{min},\infty)$ is a \emph{regular link function} if it is smooth, one-to-one and satisfies that $\Phi'(x)>0$ for all $x\in \R$ as well as
	\[ \sup_{k\ge 1}\sup_{x\in\R}\,\big|\,\Phi^{(k)}(x)\,\big|\, <\infty.  \]
\end{defin}

\begin{thm}[Convergence of PLS estimators with $\ell^1$ penalty]\label{div-maprate}
    Let $K_{min}$, $\alpha$, $\mathcal F$, $u_f$, $f_0$, $\delta_\eps$, $\mathcal F^{(r)}$ and $\theta$ be as in Theorem~\ref{div-fwdrate}. Assume in addition that $\alpha\ge 9d/2 - 1$ and that $\Phi:\R\to (K_{min},\infty)$ is a regular link function.
    
    
    \textbf{i)} For all $\lambda, \eps>0$, almost surely under $P^\eps_{f_0}$, there exists a maximizer $\fpls$ of $\mathcal I_{\lambda,\eps}$ over $\mathcal F_{\alpha,K_{min}}$.
    
    \textbf{ii)} For any $r>0$, there exist constants $C,C'>0$ such that with $\lambda\equiv\lambda_\eps\equiv C\delta_\eps^2$ and any $0<\eps<1$, we have
    \[ \sup_{f_0\in\mF^{(r)}}E^\eps_{f_0}\big\|u_{\fpls}-u_{f_0}\big\|_{L^2(\mathcal O)}^2\leq C' \delta_\eps^2.
    \]
    
    \textbf{iii)} Assume in addition $g\ge g_{min}$ for some $g_{min}>0$ and choose $\lambda$ as in part \textbf{ii)}. Then for every $r>0$ there exists $C''>0$ such that
    \[ \sup_{f_0\in\mF^{(r)}}E^\eps_{f_0}\big\|\fpls-f_0\big\|_{L^2(\mathcal O)}^2\leq C'' \delta_\eps^{2\theta}.
    \]
    
    \noindent
    Moreover, $C>0$ can be chosen independently of $r>0$, and if in addition $\alpha> 9d/2 - 1$, then we may choose $C=1$.
\end{thm}


\subsection{Steady-state Schr\"odinger equation}\label{sec-schr}
In our second example, $u_f$ arises as the solution to the following steady-state Schr\"odinger equation \cite{CZ95}
\begin{equation}\label{schr}
\begin{cases}
\frac \Delta 2 u - fu = 0 ~~~\text{on}~ \mO, \\ 
u =g~~~ \text{on}~ \partial\mO. 
\end{cases}
\end{equation}
Again, we assume $g\ge g_{min}>0$ to be a smooth and \textit{known} function this time on the boundary $\partial O$, $g\in C^\infty(\partial \mO)$,  and $f$ represents the unknown `potential' function. Whenever $f\ge 0$ and $f\in C^{\eta}(\mO)$ for some $\eta>0$, standard results from the Schauder PDE theory imply that there exists a unique solution $u=u_f\in C^{2+\eta}(\mO)\cap C(\bar{\mO})$ to (\ref{schr}), see, e.g., Chapter 6 in \cite{GT98}. For various statistical questions that have recently been studied in this non-linear statistical inverse problem, see, e.g., \cite{NVW18,N20,NW20}.

\begin{thm}[Contraction rates]\label{thm:schr}
    For $K_{min}\in [0,1)$ and integer $\alpha > 2+d$, let $\mathcal F$, $u_f$ be given by (\ref{Fdef}), (\ref{schr}) respectively, and suppose $f_0\in \mF$. Moreover, for $\de=\eps^{\frac{2\alpha+4}{2\alpha+4+d}}$, let the prior $\Pi_\eps^f\equiv \Pi^f$ be given by (\ref{f-prior}) with scaling constant $\rho=\eps^2/\de^2=\eps^{2d/(2\alpha+4+d)}$. Then the following hold.
	
	
    \textbf{i)} There exists a large enough constant $L>0$ as well as $c>0$ such that as $\eps\to0$,
	\begin{equation*}
		\pe^f \big( f: \|u_f - u_{f_0}\|_{\ltwo} \ge L\de \big| Y_\eps\big) = O_{P_{f_0}^\eps}\big( e^{-c\delta_\eps^2/\eps^2}\big).
	\end{equation*}
	\textbf{ii)} Moreover, for some $L, c>0$, as $\eps\to0$
	\begin{equation*}
		\pe^f \big( f: \|f - f_0\|_{\ltwo} \ge L\de^{\theta} \big| Y_\eps\big) = O_{P_{f_0}^\eps}\big( e^{-c\de^2/\eps^2} \big),~~~\textnormal{where}~\theta = \frac{\alpha-d-1}{\alpha-d+1}.
	\end{equation*}
	 For any $r>0$, these constants can be chosen uniformly over $\mathcal F^{(r)}:=\{f\in \mathcal F :\|\Phi^{-1}\circ f\|_{B^\alpha_{11}(\mathcal O)}\le r\}$.
\end{thm}

Similarly as for the divergence form equation (\ref{div}), this theorem is derived by verifying that the parameter-to-solution map $F\mapsto u_{\Phi\circ F}$ satisfies the local Lipschitz property (\ref{eq:local-lip}) with $\kappa=2$, whence an application of Theorem~\ref{thm:gen-contraction} yields part \textbf{i)}, and part \textbf{ii)} again is obtained using suitable PDE stability estimates. For details, see Section A of the Supplement.


For MAP-type estimators we have the following result.

\begin{thm}[Convergence of PLS estimators with $\ell^1$ penalty]\label{schr-maprate}
    Let $K_{min}$, $\alpha$, $\mathcal F$, $f_0$, $\delta_\eps$, $\mathcal F^{(r)}$ and $\theta$ be as in Theorem~\ref{thm:schr}, and suppose that $\mathcal I_{\lambda, \eps}$ is given by (\ref{I-def}) with $u_f$ from (\ref{schr}). Assume in addition that $\alpha\ge 9d/2 - 2$ and that $\Phi:\R\to (K_{min},\infty)$ is a regular link function.
    
    
    \textbf{i)} For all $\lambda, \eps>0$, almost surely under $P^\eps_{f_0}$ there exists a maximizer $\fpls$ of $\mathcal I_{\lambda,\eps}$ over $\mathcal F_{\alpha,K_{min}}$.
    
    \textbf{ii)} For any $r>0$, there exist constants $C,C'>0$ such that with $\lambda\equiv\lambda_\eps\equiv C\delta_\eps^2$ and any $0<\eps<1$, we have
    \[ \sup_{f_0\in\mF^{(r)}}E^\eps_{f_0}\big\|u_{\fpls}-u_{f_0}\big\|_{L^2(\mathcal O)}^2\leq C' \delta_\eps^2.
    \]
    
    \textbf{iii)} With $\lambda$ chosen as in part \textbf{ii)}, for every $r>0$ there exists $C''>0$ such that
    \[ \sup_{f_0\in\mF^{(r)}}E^\eps_{f_0}\big\|\fpls-f_0\big\|_{L^2(\mathcal O)}^2\leq C'' \delta_\eps^{2\theta}.
    \]
    
    \noindent
    Moreover, $C>0$ can be chosen independently of $r>0$, and if $\alpha> 9d/2 - 2$, then we may set $C=1$.
    \end{thm}

\subsection{Remarks}
\begin{rem}[Minimax rates for $u_{f_0}$]\label{rem:minimax}\normalfont
The rates of convergence $\delta_\eps$ obtained for the `PDE-constrained regression problem' of estimating $u_{f_0}$ in this section are equal to the minimax rate for estimating a function of (Besov) smoothness $\alpha+1$ (div. form problem) and $\alpha+2$ (Schr\"odinger problem) respectively (see, e.g., Section 6 in \cite{GN16}), and coincide with the rates achieved in \cite{NVW18, GN20} over $H^\alpha(\mathcal O)$ Sobolev bodies. In this sense, one may regard the map $f\mapsto u_f$ as $\kappa$-times ($\kappa=1,2$) smoothing. In \cite{NVW18} (see also \cite{N20}), it was shown that those rates are indeed minimax-optimal over Sobolev bodies. 
Conceivably, one could demonstrate analogously that $\delta_\eps$ is also optimal over $\ba$-type Besov bodies. We decided to leave this question open for future work.
\end{rem}

\begin{rem}[Truncated priors]\label{rem:trunc}\normalfont
It is natural also to consider contraction rates for `high-dimensional' (random or deterministic) truncations of non-parametric sequence priors, see, e.g., \cite{GN20,HK22,NW20} where this is pursued for Gaussian processes. Indeed, in the current context of \textit{rescaled} prior measures this leads to improved rates of contraction for the \textit{inverse problem} of recovering $f$, since prior and hence also posterior draws then belong to spaces of higher regularity than $C^b,~ b<\alpha-d$. In turn, this permits `improved' interpolation arguments in the proofs of the Theorems~\ref{div-fwdrate} \textit{\textbf{ii)}} and~\ref{thm:schr} \textit{\textbf{ii)}} akin to \cite{GN20}. In the Schr\"odinger problem one \textit{may} expect this to achieve the minimax-optimal rate for estimating $f$ (see \cite{NW20, HK22} for the Sobolev case). In the divergence form model (\ref{div}) the minimax rate for estimating $f$ is yet unknown, cf. \cite{NVW18}. An extension of our results to high-dimensional priors is an interesting future direction.
\end{rem}

\section{Lower bounds for Gaussian priors with direct observations}\label{sec-LB}

Sections~\ref{sec:genres} and~\ref{sec:pderes} demonstrated that Laplace priors achieve certain rates of contraction for truth values belonging to $\ell^1$-type Besov bodies $B^\alpha_{11}$. While intuitively Laplace-type priors seem more suited towards modelling $B^\alpha_{11}$-functions than Gaussian priors, which in turn seem more inherently compatible with $\ell^2$-type Sobolev regularity, it has been an open question whether Gaussian priors are indeed fundamentally limited to a slower  convergence rate than Laplace priors. In Theorem~\ref{thm-lb} below, we derive such lower bounds for Gaussian process priors. 


For simplicity, we work in the standard Gaussian white noise model for nonparametric regression with \textit{direct} observations in one dimension, $\mathcal O=[0,1]$. Here, for $f_0:[0,1]\to \R$ the data are given as
\[	Y_\eps=f_0 + \eps \W, ~~~ \eps>0, \]
where $\W$ again denotes the standard Gaussian white noise process for the Hilbert space $L^2([0,1])$.

In the minimax estimation literature, Donoho and Johnstone's seminal paper \cite{DJ98} established that while linear estimators are able to achieve the minimax-optimal rate over $\ell^2$-type Sobolev classes, over the \emph{spatially inhomogeneous} Besov function classes $B^\alpha_{pq}\equiv B^{\alpha}_{pq}([0,1])$ with $1\leq p<2$, the best achievable convergence rate for linear estimators is \textit{slower}, by a polynomial factor, than the minimax rate. Specifically, Theorem 1 in \cite{DJ98} shows that if either $\alpha>1/p$, $1\le p< 2, 1\le q<\infty$ or $\alpha=p=q=1$, then the `linear' minimax rate is given by
\begin{equation}\label{lin-rate}
    \inf_{\hat f(Y_\eps)~\textnormal{linear}}~\sup_{f_0:\|f_0\|_{B^{\alpha}_{pq}}\le 1}E^\eps_{f_0}\|\hat f-f_0\|_{L^2}^2\simeq l_\eps^2,~~~\text{where}~l_\eps \coloneqq \eps^{\frac{2\alpha-(2-p)/p}{2\alpha+1-(2-p)/p}},
\end{equation}
while the `overall' minimax rate is given by
\begin{equation}\label{real-rate}
    \inf_{\hat f(Y_\eps)}~\sup_{f_0:\|f_0\|_{B^{\alpha}_{pq}}\le 1}E^\eps_{f_0}\|\hat f-f_0\|_{L^2}^2\simeq m_\eps^2,~~~\text{where}~m_\eps\coloneqq \eps^{\frac{2\alpha}{2\alpha+1}}.
\end{equation}

In line with this, when $f_0\in B^\alpha_{pp}$ for $1\le p<2$, the (upper bounds for) contraction rates from Theorem 5.9 in \cite{ADH21} display a gap between Laplace priors and Gaussian priors (which, with direct observations, constitute a `linear procedure'). In particular, Laplace priors can be tuned to achieve the 'overall' minimax rate $m_\eps$ (up to log-terms if $p>1$), while Gaussian priors appear to be limited by the linear-minimax rate $l_\eps$; see Section H in the supplement of \cite{ADH21} for a discussion. Our result confirms that this is \textit{not} an artifact of the proofs in \cite{ADH21}.

Let $\mathcal P_{GPP}$ denote the collection of all Gaussian process priors supported in $L^2([0,1])$.




\begin{thm}[Lower bound for GPPs]\label{thm-lb}
	Suppose that either $\alpha>1/p$, $1\le p< 2$ and $1\le q<\infty$ or $\alpha=p=q=1$, and let $\{\delta_\eps:0<\eps<1\}\subseteq \R_{>0}$ be some decreasing sequence as $\eps \to 0$. Assume moreover that $(\Pi_\eps: 0<\eps<1)\subseteq \mathcal P_{GPP}$ is a sequence of mean-zero Gaussian process priors supported on $L^2([0,1])$, such that for all $\eta>0$ we have the posterior contraction bound
	\begin{equation}\label{unif-contr}
		\sup_{f_0:\|f_0\|_{B^\alpha_{pq}}\le 1}P_{f_0}^\eps\big(\Pi_\eps\big( f: \|f-f_0\|_{L^2}\ge \delta_\eps \big| Y_\eps \big)\ge \eta \big) \xrightarrow{\eps\to 0} 0.
	\end{equation}
	 Then there exists some constant $c>0$ such that for the linear minimax rate $l_\eps=\eps^{\frac{2\alpha-(2-p)/p}{2\alpha+1-(2-p)/p}}$ given by (\ref{lin-rate}), we have \[ \delta_\eps \ge cl_\eps,~~ 0<\eps<1. \]
\end{thm}

Theorem~\ref{thm:gen-contraction} with $\mathcal G$ as the identity operator immediately implies that, for $p=q=d=1$ and $\alpha$ large enough, Laplace priors fulfill the contraction statement (\ref{unif-contr}) with the `actual' minimax rate of estimation $m_\eps = o(l_\eps)$ from (\ref{real-rate}) in place of $\delta_\eps$. This indicates that Laplace priors, from an information-theoretic minimax viewpoint, should be favoured over Gaussian priors in estimation problems over $\ell^1$-type Besov bodies. [For $1<p<2$, analogous statements can be derived using Theorem 5.9 in \cite{ADH21}.]

The paper \cite{IC08} by I. Castillo previously studied lower bounds for rates of contraction for GPPs $\Pi$ on Banach spaces $\mathbb B$. However, the approach taken there is fundamentally different from ours, since it bases on general abstract conditions about the `concentration function' 
of $\Pi$ [playing a key role in the derivation of upper bounds for rates of contraction for GPPs in \cite{VV08}], which can be characterized for \textit{specific} commonly used priors. Thus, in order to study lower bounds over \textit{all} Gaussian priors simultaneously, a proof technique which uses the aforementioned minimax estimation theory over Besov classes seems more naturally suited.

\begin{proof}[Proof of Theorem \ref{thm-lb}]
    Let $M_\eps$ denote the center of the smallest $L^2$-ball possessing posterior probability at least $1/2$. By the Gaussianity of the posterior and Anderson's lemma (cf. Section 2.4.1 of \cite{GN16}) $M_\eps$ equals the posterior mean, which in turn can be realized as $M_\eps= \mathcal A Y_\eps\in L^2(\mathcal O)$, where $\mathcal A$ is a linear (trace-class) operator, see, e.g., \cite[Theorem 3.2]{LZ00}. Let $\eta\in (0,1/2)$. Noting that on the event $\{\Pi_\eps\big( f: \|f-f_0\|_{L^2}\ge \delta_\eps \big| Y_\eps \big)< \eta \}$ the radius of that ball (around $M_\eps$) is at most $\delta_\eps$, we obtain from (\ref{unif-contr}) that 
	\begin{equation}\label{close-to-0}
		\begin{split}
			\sup_{f_0:\|f_0\|_{B^{\alpha}_{pq}} \le 1}&P_{f_0}^\eps\big(\|M_\eps-f_0\|_{L^2}\ge 2\delta_\eps\big)\\
			&\le  \sup_{f_0:\|f_0\|_{B^{\alpha}_{pq}} \le 1}P_{f_0}^\eps\Big(\Pi_\eps\big( f: \|f-f_0\|_{L^2}\ge \delta_\eps \big| Y_\eps \big)\ge \eta \Big) \xrightarrow{\eps\to 0} 0.
		\end{split}
	\end{equation}
	
	For the rest of the proof, let us fix $f_0\in B^{\alpha}_{pq}$ with norm at most $1$. Unless otherwise noted, the constants featuring below can be chosen uniformly over such $f_0$. Noting that 
	\[ \|M_\eps-f_0\|_{L^2}= \sup_{\psi:\|\psi\|_{L^2}\le 1}|\langle M_\eps-f_0, \psi \rangle|, \]
	for every $\rho>0$ there exists $\eps_\rho>0$ such that
	\begin{equation}\label{small}
        \sup_{\eps\le\eps_\rho,~\psi:\|\psi\|_{L^2}\le 1} P_{f_0}^\eps\big( \big|\langle M_\eps-f_0, \psi\rangle \big|\ge 2\delta_\eps \big) \le \sup_{\eps\le\eps_\rho} P_{f_0}^\eps\big( \|M_\eps-f_0\|_{L^2}\ge 2\delta_\eps \big)\le \rho.
	\end{equation}
	Since $\langle M_\eps-f_0, \psi\rangle $ is Gaussian, it follows that there exists a sequence $c_\eps\xrightarrow{\eps\to 0} 0$ such that \[\sigma^2_\eps:=\sup_{\psi:\|\psi\|_{L^2}\le 1}\text{Var}(\langle M_\eps-f_0, \psi\rangle)\le c_\eps \delta_\eps^2.\]
	Indeed, if this were not true, there would exist a subsequence $(\eps_l:l\ge 1),~\eps_l\to 0$ and a collection of functions $(\psi_l:l\ge 1)$ such that for some small constant $\ubar c$, $\text{Var}(\langle M_{\eps_l}-f_0, \psi_l\rangle)\ge \ubar c \delta_{\eps_l}^2$, for all $l\geq1$. Thus, there would exist another sufficiently small constant $\ubar c '=\ubar c '(\ubar c)>0$ such that for any $\zeta>0$,
	\begin{equation}
	    \begin{split}
	       \sup_{\eps\le\zeta,\psi:\|\psi\|_{L^2}\le 1}& P_{f_0}^{\eps}\big( \big|\langle M_\eps-f_0, \psi\rangle \big|\ge 2\delta_\eps \big) \ge \inf_{l\ge 1} P_{f_0}^{\eps_l}\Big( \frac{1}{\sqrt{\ubar c}\delta_{\eps_l}}\big|\langle M_{\eps_l}-f_0, \psi_l\rangle \big|\ge 2/\sqrt{\ubar c} \Big)\ge \ubar c',
	    \end{split}
	\end{equation}
	which contradicts (\ref{small}).
	
	It follows from the Borell-Sudakov-Tsirelson concentration inequality (see, e.g., Theorem 2.5.8. in \cite{GN16}) that 
	\begin{equation}\label{BST}
		P_{f_0}^\eps\big( \big| \|M_\eps-f_0\|_{L^2}- E_{f_0}^\eps\|M_\eps-f_0\|_{L^2} \big|\ge u \big) \le 2e^{-\frac{u^2}{2c_\eps \delta_\eps^2}}, ~~~ u>0.
	\end{equation}
	Choosing $u=\delta_\eps$, the above yields
	\begin{equation}\label{close-to-mean}
		\sup_{f_0:\|f_0\|_{B^\alpha_{pq}}\le 1}P_{f_0}^\eps\big( \big| \|M_\eps-f_0\|_{L^2}- E_{f_0}^\eps\|M_\eps-f_0\|_{L^2} \big|\ge \delta_\eps \big) \xrightarrow{\eps\to 0} 0.
	\end{equation}
	Combining (\ref{close-to-0}) and (\ref{close-to-mean}) yields that for some $c'>0$, $M_\eps$ satisfies
	\[ \sup_{f_0:\|f_0\|_{B^\alpha_{pq}}\le 1} E_{f_0}^\eps\|M_\eps-f_0\|_{L^2} \leq c' \delta_\eps. \]
	Using this risk bound as well as (\ref{BST}), we then also obtain a bound for the mean squared error too -- indeed, denoting the random variable $S_\eps=\|M_\eps-f_0\|_{L^2}$, we have 
	\begin{equation*}
	\begin{split}
	    E_{f_0}^\eps[S_\eps^2]&\le 2 \int_0^{2c'\delta_\eps}tP_{f_0}^\eps(S_\eps\ge t) dt + 2 \int_{2c'\delta_\eps}^{\infty}tP_{f_0}^\eps(S_\eps - E^\eps_{f_0}[S_\eps] \ge t- c'\delta_\eps ) dt\\
	    &\lesssim \delta_\eps^2 + \int_{2c'\delta_\eps}^{\infty} 	t	P_{f_0}^\eps\big( \big|S_\eps -E_{f_0}^\eps [S_\eps]\big| \ge t /2 \big)dt \lesssim \delta_\eps^2 + \int_{2c'\delta_\eps}^{\infty} t\exp\Big(-\frac{t^2}{8c_\eps\delta_\eps^2}\Big)dt\\
	    &\lesssim \delta_\eps^2 + \delta_\eps^2 \int_{2c'}^{\infty} u \exp\Big(-\frac{u^2}{8c_\eps}\Big)du\lesssim \delta_\eps^2.
	    \end{split}
	\end{equation*}
	Since $M_\eps$ is a linear estimator (in $Y_\eps$) and since the above inequality holds uniformly over $f_0\in \{f\in B^\alpha_{pq}:\|f\|_{B^\alpha_{pq}}\le 1 \}$, the claim now follows from the linear minimax rate characterised in Theorem 1 in \cite{DJ98}, where we have replaced $n$ there by $\eps^{-2}$ here.
\end{proof}

Note that for general non-linear inverse problems, whether Gaussian priors are limited to a polynomially slower than minimax contraction rate may depend subtly on the mapping properties of the forward map; even if $F_0$ is  spatially inhomogeneous, the image $\mathcal G(F_0)$ may not be. Conceivably, under additional assumptions one may extend our results to \textit{certain} inverse problems. For instance, a set of sufficient conditions would be for $\mathcal G$ to be a \textit{linear} map satisfying isomorphism properties of the type $C^{-1}\|\mathcal G(F) \|_{B^{\alpha+\kappa}_{11}}\le \|F \|_{B^{\alpha}_{11}}\le C\|\mathcal G(F) \|_{B^{\alpha+\kappa}_{11}},~(C>0)$ on the $p=q=1$ Besov scale (cf. \cite{KVV11} for similar assumptions on the Sobolev scale), but we do not pursue this direction further here.

\section{Proofs for Section~\ref{sec:genres}}\label{sec:proofsfor2}
\subsection{Proof of Theorem~\ref{thm:gen-contraction}}\label{sec:24proof}

	

Let $\peg$ and $\peg(\cdot|Y_\eps)$ denote the respective push-forward distributions of the prior $\pe$ and posterior $\pe(\cdot|Y_\eps)$ under the forward map $\mG$, both of which are supported on $\mG(\vi)\subset \ltwod$. We prove Theorem~\ref{thm:gen-contraction} by using contraction rate theory for nonparametric Bayes procedures \cite{GGV00,GV17}, specifically in the form of Theorem 7.3.5 in \cite{GN16} (see also Theorem 28 in \cite{N20}). Since our statement of interest regards the recovery of $\mG(F_0)$, we will utilise this result with $\Pi=\peg$ (and with $1/\eps^2$ in place of $n$), for which we verify the hypotheses (7.95)-(7.97) in \cite{GN16}. {Note that we follow \cite{N20,AN19} in constructing statistical tests directly for the $L^2$-distance (equal to the intrinsic `covariance metric' for the Gaussian white noise process), instead of appealing to a general Hellinger testing approach which has been employed in i.i.d. (random design) regression settings \cite{GN20,MNP19b}.} We also note that to obtain the quantitative posterior contraction bound in Theorem~\ref{thm:gen-contraction}, our proof below verifies a more precise type-I error bound than is required in (7.97) of \cite{GN16}. 
	
First, Lemma~\ref{lem:forwardpriormass} implies that there exist constants $B>1,C>0$ such that for $\tde:=B\delta_\eps$, we have
\begin{equation}\label{eq:sbg}\peg\big(\mG(F) : \norm{\mG(F)-\mG(F_0)}_{\ltwod}\leq\tde\big)\geq e^{-C \tde^2/\eps^2},\end{equation}
	which verifies (7.95) in \cite{GN16}. Second, choosing $\mM_\eps\coloneqq\mG(\mAe)$ with $\mAe$ from (\ref{eq:mat}) below, we see that the `sieve set' condition in (7.96) from \cite{GN16} is also fulfilled. Indeed, by Lemma~\ref{lem:sieve} there exists some (sufficiently large) $M>1$ in the definition of $\mAe$ as well as some $L'>0$, such that
	\begin{equation}\label{eq:sieveg}\peg(\mM_\eps^c)=\pe(\mAe^c)\leq L' e^{-(C+4)\tde^2/\eps^2}.\end{equation}

	
	It remains to verify the existence of statistical tests $\Psi=\Psi(Y_\eps)$ satisfying
	\begin{equation}\label{eq:exam}
		P_{F_0}^\eps( \Psi =1 ) + \sup_{\varphi\in \mM_\eps: \|\varphi -\mG(F_0) \|_{\ltwod}\ge \cons\tde } P_{F}^\eps( \Psi =0 )\le L'e^{-(C+4)\tder}, 
	\end{equation}
	where $\cons>0$ is some constant to be chosen below. To construct those tests, denote $\kmap:=\supp(\chi)$ (with $\chi$ as in Assumption~\ref{ass:truth}), and for $B=\norm{F_0}_{\ba}\vee M^2$ (where $M$ is as in the definition of $\mAe$), define $\mathcal J=\mathcal J_{\lambda,\eps}$ and $\tv_B$ as in \eqref{Jdef} and \eqref{eq:vb} below, respectively. Then, set
	\begin{equation*}
		\begin{split}
			\Psi(Y_\eps):=\mathbbm{1} \Big(\| \mG(\hat F)-\mG(F_0) \|_{\ltwod} \ge \cons\tde / 2 ~~\text{for some}~~\hat F\in \arg\max_{F\in \tv_B}\mathcal J(F)\Big).
		\end{split}
	\end{equation*}

	To control the first term in (\ref{eq:exam}), noting that $\supp(F_0)\subseteq K\subset\tilde K$, we apply Theorem~\ref{thm:plsrbddnorm} below with $F_\dagger=F_*=F_0$ to see that there exist sufficiently large constants $D', C'>0$ such that
	\begin{align}\label{type-I-bound}
	P_{F_0}^\eps\Big(\exists \hat{F}&\in \arg\max_{F\in \tv_B}\mathcal J(F) \;\text{s.t.}\; \norm{\mG(\hat F)-\mG(F_0)}_{\ltwod}\geq D'\tde/2\Big)\leq C' e^{-(C+4)\tder}.
	\end{align}

	To control the type-II errors in (\ref{eq:exam}), let $F\in \mA_\eps$ be such that $\norm{\mG(F) -\mG(F_0)}_{\ltwod}\ge \cons\tde$ [for $\cons$ yet to be chosen], and let $F=F_1+F_2$ be the decomposition from \eqref{eq:mat}. By definition of $\mAe$, using the continuous embedding $ \ba\subseteq \cb$ and noting that $\sup_{F\in \mAe}\|F_2\|_{B^\alpha_{11}(\mO)}<\infty$, the local Lipschitz estimate (\ref{eq:local-lip}) yields
	\begin{equation*}
	    \begin{split}
	        \sup_{F\in\mAe}\norm{\mG(F_2)-\mG(F)}_{\ltwod}
	        \lesssim \sup_{F\in\mAe}\|F_1\|_{\hk}\lesssim \delta_\eps.
	    \end{split}
	\end{equation*} 
	Recalling the definition \eqref{eq:tals} of $\tals$ with $\lambda\simeq \de^2$, we thus obtain
	\[ \sup_{F\in \mAe} \tau_\lambda^2(F_2,F) \lesssim \delta_\eps^2. \]
	Using this, and by another application of Theorem~\ref{thm:plsrbddnorm}, this time with $F_*=F_2$ and $F_\dagger=F$ [and again $\kmap=\supp(\chi)$], there exists some $\conu>0$ such that
	\begin{equation*}
	\sup_{F\in \mAe}P_{F}^\eps\Big(\exists \hat{F}\in \arg\max_{F\in \tv_B}\mathcal J(F)~\text{s.t.}~\norm{\mG(\hat F)-\mG(F)}_{\ltwod}\geq \conu\tde \Big)\lesssim e^{-(C+4)\tder}.
	\end{equation*}

	Now, choosing $\cons > \max\{D',2 \conu\}$ yields that for any $F\in\mAe$ with
	$\norm{\mG(F) -\mG(F_0)}_{\ltwod}\ge \cons\tde$,
	\begin{align*}
	P_F^\eps&(\Psi(Y_\eps)=0)=
P_F^\eps\Big(\norm{\mG(\hat{F})-\mG(F_0)}_{\ltwod}< \cons\tde /2~\text{for all}~\hat{F}\in \arg\max_{F\in \tv_B}\mathcal J(F)\Big)\\
&\leq P_F^\eps\Big(\exists \hat{F}\in \arg\max_{F\in \tv_B}\mathcal J(F)~\text{s.t.}~\norm{\mG(\hat F)-\mG(F)}_{\ltwod}\geq \cons\tde/2 \Big)\lesssim e^{-(C+4)\tder},
\end{align*}
which, together with (\ref{type-I-bound}), proves the desired inequality \eqref{eq:exam} [possibly suitably enlarging $L'>0$ from (\ref{eq:sieveg})]. Summarizing, the displays \eqref{eq:sbg}, \eqref{eq:sieveg} and \eqref{eq:exam} verify the hypotheses of Theorem 7.3.5 in \cite{GN16}, yielding that for some $L>0$ [specifically $L=DB$],
\begin{equation*}
	\begin{split}
		\pe(F:&\norm{\mG(F)-\mG(F_0)}_{\ltwod}>L\de|Y_\eps)=\peg\big(\mG(F):\|\mG(F)-\mG(F_0)\|_{\ltwod}>L\de\big|Y_\eps\big)\xrightarrow{\eps\to 0}0,
	\end{split}
\end{equation*}
in $P_{F_0}^\eps$-probability.


The more quantitative convergence statement (\ref{eq:forwardcontraction}) follows from the proof of Theorem 7.3.5 in \cite{GN16}, which in turn is analogous to the proof of Theorem 7.3.1 \cite{GN16}. Indeed, an inspection of the latter proof shows that i) we may choose $\epsilon$ in the first display on p.575 of \cite{GN16} as $e^{-\delta_\eps^2/\eps^2}$, ii) our testing bound (\ref{type-I-bound}) can be used to bound the type-I error in (7.89) of \cite{GN16} by the order of $e^{-\delta_\eps^2/\eps^2}$, iii) the dominating probability in the proof of Theorem 7.3.5 arises from the `small ball estimate' in Lemma 7.3.4 of \cite{GN16}, which in our case is of the order $\eps^2/\delta_\eps^2\to 0$ as $\eps\to 0$. [We may thus choose $\mu = 2d/(2\alpha+2\kappa +d)$ in Theorem~\ref{thm:gen-contraction}.]

Finally we note that by the quantitative arguments from the preceding paragraph as well as \eqref{eq:sbg} and \eqref{eq:sieveg}, there exists $\mu>0$ such that $\mAe$ also satisfies the quantitative estimate
\[ P_{F_0}^\eps\big(\pe(\mAe^c\,\big|\,Y_\eps) \ge e^{-\delta_\eps^2/\eps^2} \big)\lesssim \eps^\mu,\]
which by the definition of $\mAe$ implies \eqref{eq:Cbconcentration}.
\qed

In the remainder of this section, we prove the auxiliary Lemmas~\ref{lem:forwardpriormass} and~\ref{lem:sieve}.


\begin{lem}\label{lem:forwardpriormass}
Let $\pe, \de$ and $F_0$ be as in Assumption~\ref{ass:truth}, and $\mG$ as in Assumption~\ref{ass:kgb}. Then for any $B>0$ sufficiently large, there exists $C>0$  such that 
\begin{equation}
\pe(F:\norm{\mG(F)-\mG(F_0)}_{\ltwod}\leq B\de)\geq e^{-C \de^2/\eps^2},
\end{equation}
for all $\eps>0$ small enough. 
The constant $C$ only depends on $F_0$ via $\|F_0\|_{B^\alpha_{11}}$; in particular it can be chosen uniformly over $F_0\in \vi \cap \{F\in B^\alpha_{11}: \|F\|_{B^{\alpha}_{11}}\le r,~\textnormal{supp}(F)\subseteq K \}$, $r>0$.
\end{lem}

\begin{proof}
 Let $B>0$ be fixed. Let $1\le\beta<\alpha-d$ as in Assumption~\ref{ass:kgb}, and notice that under Assumption~\ref{ass:truth}, $F_0\in \cb$. Let $M>\norm{F_0}_{\cb}\vee 1$ be a fixed constant, then by Assumption~\ref{ass:kgb} with $R=2M$, there exists $\crm>0$ sufficiently small such that 
\begin{align*}
\pe(F:\norm{\mG(F)-\mG(F_0)}_{\ltwod}\leq B\de)&\geq \pe(F:\norm{F-F_0}_{\hk}\leq B \crm\de, \norm{F-F_0}_{\cb}\leq M)\\
&=\pe(F:F-F_0\in S_1\cap S_2),
\end{align*}
where $S_1\coloneqq\{F: \norm{F}_{\hk}\leq B \crm\de\}\quad\text{and}\quad S_2\coloneqq\{F:\norm{F}_{\cb}\leq M\}.$

Combining with the definition of $\pe$, the fact that $\chi F_0=F_0$ (see Definition~\ref{def:prior} and Assumption~\ref{ass:truth}) and  Lemma \ref{lem:chibd} in the Supplement (asserting that $\norm{\chi F}_{\hk}\lesssim\norm{F}_{\hkr}$ and $\norm{\chi F}_{\cb}\lesssim\norm{F}_{\cbr}$), we find that there is some sufficiently small $c'>0$ such that
\begin{align*}
\pe(F:\norm{\mG(F)-\mG(F_0)}_{\ltwod}\leq B\de)
\geq\tp\Big(F: F-\der F_0\in \der(\tc_1\cap \tc_2)\Big),
\end{align*}
where $\tc_1\coloneqq\{F: \norm{F}_{\hkr}\leq c'B \crm\de\}\quad\text{and} \quad\tc_2\coloneqq\{F:\norm{F}_{\cbr}\leq c'M\}.$

Using Lemma~ \ref{lem:shsb} and display \eqref{eq:tom}, and the fact that $\norm{F_0}_{\tom}\lesssim \norm{F_0}_{\ba}$ (since $F_0$ is supported in $\mO$), we obtain that for some constant $\tilde{c}>0$,
\begin{align}\label{eq:C1C2}
\pe(F:\norm{\mG(F)-\mG(F_0)}_{\ltwod}\leq B\de)&\geq e^{-\frac{\tilde{c}\de^2}{\eps^2}\norm{F_0}_{\ba}}\tp\Big(\der(\tc_1\cap \tc_2)\Big)\nonumber\\
&\geq e^{-\frac{\tilde{c}\de^2}{\eps^2}\norm{F_0}_{\ba}}\big[\tp\Big(\der\tc_1\Big)-\tp\Big((\der\tc_2)^c\Big)\big].
\end{align}

We first compute a lower bound for $\tp(\der\tc_1)$. We have 
\[\tp\Big(\der\tc_1\Big)=\tp\big(F: \norm{F}_{\hkr}\leq c'B\crm \de^3 /\eps^2\big),\quad\text{where}~~
\de^3/\eps^2=\eps^{\frac{2\kappa+2\alpha-2d}{2\kappa+2\alpha+d}}.\] 
Since Assumption~\ref{ass:truth} implies that $d<\kappa+\alpha$, we have that $\de^3/\eps^2\to0$ as $\eps\to0$, whence it follows from Lemma~\ref{lem:sb} below that
\[-\log\tp\Big(\der\tc_1\Big)\lesssim \big(c'B \crm \de^3/\eps^2\big)^{-\frac{d}{\kappa+\alpha-d}}, \quad\text{where}~~
\big(\de^3/\eps^2\big)^{-\frac{d}{\kappa+\alpha-d}}=\de^2/\eps^2\to\infty.\]
Therefore, for some (potentially different) $c'>0$, we have that
\begin{equation}
\label{eq:C1}\tp\Big(\der\tc_1\Big)\geq \exp\big(-(c'B \crm)^{-\frac{d}{\kappa+\alpha-d}}\de^2/\eps^2\big).
\end{equation}
To upper bound $\tp\big((\der \tc_2)^c\big)$, we use Lemma \ref{lem:fernique} below to see that there exist $\cona, \conb>0$ (large and small enough, respectively) such that
\begin{equation}\label{eq:C2}
\tp\Big((\der \tc_2)^c\Big)\leq \cona e^{-\conb c' M\de^2/\eps^2}.
\end{equation}

Combining \eqref{eq:C1C2}, \eqref{eq:C1} and \eqref{eq:C2}, we find that by choosing $B$ sufficiently large, there exists $C>0$ such that
\begin{align*}
\pe(F&:\norm{\mG(F)-\mG(F_0)}_{\ltwod}\leq B\de)\\
&\gtrsim e^{-\tilde{c}\der\norm{F_0}_{\ba}}\big(e^{-(c'B \crm)^{-\frac{d}{\kappa+\alpha-d}}\de^2/\eps^2}- e^{-\conb c' M\de^2/\eps^2}\big)\gtrsim e^{-C \de^2/\eps^2}.
\end{align*}
Finally, we may enlarge $C$ so that for $\eps$ small enough, the last bound holds with unit constant in front of the exponential.
\end{proof}


\begin{lem}\label{lem:sieve}
Let $\pe, \de$ and $\chi$ be as in Assumption~\ref{ass:truth}. 
For any $C>0$, there exist sufficiently large constants $M, L>1$, such that for all $\eps>0$ small enough, the set
\begin{equation}\label{eq:mat}
	\begin{split}
	\mAe\coloneqq\Big\{F=F_1+F_2: ~ \norm{F_1}_{\hk}\leq \delta_\eps, ~\norm{F_2}_{\ba}\leq M^2, ~\norm{F}_{\cad}\leq M,~~~&\\ \supp(F_i)\subseteq\supp(\chi), i=1,2&\Big\}
	\end{split}
\end{equation}
satisfies
\begin{equation}\label{eq:sievemass}
\pe(\mAe^c)\leq L e^{-C\de^2/\eps^2}.
\end{equation}
\end{lem}

\begin{proof}
We first note that 
\[\pe(\mA_{\eps}^c)\leq \pe(\mR_\eps^c)+\pe(S_2^c),\]
where $S_2=\{F:\norm{F}_{\cb}\leq M\}$ is as in the proof of Lemma~\ref{lem:forwardpriormass} for $\beta=\alpha-d-1$, and where $\mR_\eps$ is defined as
\begin{align}\label{eq:sieve}
\mR_\eps\coloneqq \{F=F_1+F_2:\;\norm{F_1}_{\hk}\leq \de, ~\norm{F_2}_{\ba}\leq  M^2, ~ \supp(F_i)\subseteq\supp(\chi), i=1,2 & \}.
\end{align}
By Lemma \ref{lem:chibd} in the Supplement and the definitions of $\tp, \pe$, and with $\tc_2$ also defined in the proof of Lemma~\ref{lem:forwardpriormass}, it follows from \eqref{eq:C2} that for $M$ large enough, 
\[\pe(S_2^c)\le\tp\Big((\der \tc_2)^c\Big)\lesssim e^{-C\de^2/\eps^2}.\] It thus remains to bound $\pe(\mR_\eps^c)$.

Again by the definitions of $\tp, \pe$ together with Lemma \ref{lem:chibd} of the Supplement, there exists a small enough constant $c>0$ such that 
\begin{align*}\pe(\mR_\eps)&=\tp\Big(F: \chi F\in \der\mR_\eps\Big)\\
&\geq \tp\Big(F:\chi F=\chi F_1+\chi F_2, ~\norm{\chi F_1}_{\hk}\leq \frac{\de^3}{\eps^2}, \norm{\chi F_2}_{\ba}\leq  M^2\der\Big)\\
&\geq \tp\Big(F= F_1+F_2:\norm{F_1}_{\hkr}\leq c\frac{\de^3}{\eps^2}, \norm{F_2}_{\bard}\leq c M^2\der, F_i\in{\rm span}\{\subcol\}, i=1,2\Big).
\end{align*}
By Lemma \ref{lem:approx} of the Supplement (which quantifies the approximation of $H^{\alpha-\frac{d}2}$ by $B^{\alpha}_{11}$ in $(H^\kappa)^\ast$-norm), there exists $c'>0$ sufficiently small such that
\begin{align*}
\pe(\mR_\eps)& \geq \tp\Big(F= F_1+F_2+F_3:~\norm{F_1}_{\hkr}\leq \frac{c\de^3}{2\eps^2}, \norm{F_2}_{\bard}\leq \frac{cM^2\de^2}{2\eps^2},\\
&~~~~~~~~~~~~~~~~\quad\qquad\qquad\norm{F_3}_{\hard}\leq \frac{c'cM\de}{2\eps}, ~ F_i\in{\rm span}\{\subcol\}, i=1,2,3\Big).
\end{align*}
Thus, for some small enough $c''>0$, an application of Lemma~\ref{lem:tal} below with $r=c''M^2\der$ and $A=\{F_1\in {\rm span}\{\subcol\}:\norm{F_1}_{\hkr}\leq \frac{c}2 \frac{\de^3}{\eps^2}\}$ yields that
\begin{equation}\label{eq:tal}
\pe(\mR_\eps^c)\leq\frac{1}{\tp(A)}\exp\Big(-\frac{c''M^2}{\Lambda}\der\Big),
\end{equation}
where $\Lambda>0$ is a fixed constant. Moreover, similarly to the computation of the lower bound on $\tp(\der \tc_1)$ in the proof of Lemma~\ref{lem:forwardpriormass}, Lemma~\ref{lem:sb} below implies that there exists a constant $k'>0$ such that 
\[\tp(A)\geq \exp\big(-(c k')^{-\frac{d}{\kappa+\alpha-d}}\de^2/\eps^2\big).\]
Using this, \eqref{eq:tal}, and choosing $M$ large enough, we obtain
\begin{equation}\label{eq:tal2}
\pe(\mR_\eps^c)\leq \exp\big(-\big(c''M^2\Lambda^{-1}-(c k')^{-\frac{d}{\kappa+\alpha-d}}\big)\de^2/\eps^2\big)\leq e^{-C\de^2/\eps^2}.
\end{equation}
\end{proof}


\subsection{Proof of Theorem~\ref{thm:plsfr}, part \textbf{ii)}}\label{sec:proof26}
The proof of part \textbf{i)} is delayed to Section \ref{sec:existproof}; let us first focus on the more involved proof of part \textbf{ii)}. To accommodate both the proof of Theorem~\ref{thm:plsfr} and that of Theorem~\ref{thm:plsrbddnorm} below, we adopt a setup which is slightly more general than that of Section~\ref{sec:PLS}. For any fixed, compact $\tilde{K}\subseteq \mO$, suppose that $\mW\subseteq \{F\in \vi\cap \ba: \supp(F)\subseteq \tilde K\}$ is a subset over which the objective $\mathcal J_{\lambda,\eps}$ is considered. We recall the functional
\[\tals(F_1,F_2)\coloneqq \norm{\mG(F_1)-\mG(F_2)}_{\ltwod}^2+\lambda\norm{F_1}_{\toma},~~~\lambda >0,\]
defined for any $F_1\in \mW$ and $F_2\in\vi$. Moreover, for any $F_\ast\in \mW$ and $\lambda, R>0$, define the sets
\begin{align}
    \label{eq:vast}
\vast(\lambda,R)&:=\{F\in \mW: \tau_\lambda^2(F,F_\ast)\leq R^2\},\\
\label{eq:dast}
\dast(\lambda,R)&:=\{\mG(F): F\in \vast(\lambda,R)\}.
\end{align}
For any normed vector space $(X,\|\cdot\|)$, any subset $A\subseteq X$ and any $\rho>0$, let $H(\rho, A, \norm{\cdot})$ denote the metric entropy of $A$ in $\norm{\cdot}$-distance [that is, $H(\rho, A, \norm{\cdot})=\log N(\rho, A, \norm{\cdot})$ where $N(\rho, A, \norm{\cdot})$ denotes the minimal number of closed $\norm{\cdot}$-balls of radius $\rho$ required to cover $A$]. Then, define
\begin{equation}\label{eq:jast}
\jast(\lambda,R)=R+\int_0^{2R}\sqrt{H(\rho, \dast(\lambda,R),\norm{\cdot}_{\ltwod})}d\rho.
\end{equation}

The following proposition regards concentration properties of penalized least squares estimators with $\ell_1$-type Besov penalty.

\begin{prop}\label{thm:genplsfr}
Suppose that data $Y_\eps$ arise from \eqref{eq:model} for some $F=F_\dagger\in\vi$. For any $F_\ast\in\mW$, let $\psast(\lambda,R)\geq \jast(\lambda,R)$ be some upper bound such that $R \mapsto \psast(\lambda,R)/R^2$ is non-increasing for $R>0$. Then, there exists a universal constant $\conf >0$ such that for all $\eps, \lambda, \delta>0$ satisfying 
\begin{equation}\label{eq:psast}
\conf\eps\psast(\lambda,\delta)\leq \delta^2
\end{equation}
and any $R\geq \delta$, 
\begin{align}
P^\eps_{F_\dagger}\Big(\tals(\hat{F}, F_\dagger)\geq 2(\tals(F_\ast, F_\dagger)+R^2) ~\text{for some}~\hat{F}\in \arg\max_{F\in \mW} & \jle (F)\Big)\leq \conf \exp\Big(-\frac{R^2}{\conf^2\eps^2}\Big).
\label{eq:plsconc}
\end{align}
Moreover, for any maximizer $\hat{F}$ of $\jle$ over $\mW$, we have for some universal constant $\conh>0$
\begin{equation}\label{eq:plsexp}
E^\eps_{F_\dagger}[\tals(\hat{F},F_\dagger)]\leq\conh(\tals(F_\ast,F_\dagger)+\delta^2+\eps^2).
\end{equation}
\end{prop}

This result is proved using techniques from M-estimation, specifically from \cite{V01}. Since it can be derived from Theorem 2.1 of \cite{V01} in essentially the same manner as Theorem 18 in \cite{NVW18} (where Sobolev penalties are considered), the proof is omitted here.



The preceding proposition suggests that the convergence bounds for $\|\mG(\hat{F})-\mG(F_\dagger)\|_{L^2(\mathcal D)}$ in Theorem~\ref{thm:plsfr} can be obtained by a suitable bound $\psast(\lambda,R)$ (with sub-quadratic dependence on $R$) on the entropy integral $\jast(\lambda,R)$. For any given regularization parameter $\lambda=\lambda(\eps)$, the requirement \eqref{eq:psast} dictates the best possible choice $R^2=\delta^2$ in \eqref{eq:plsconc}; typical choices of $\lambda$ would balance $\delta^2$ with the size of the `approximation term' $\tals(F_\ast,F_\dagger)$, which in turn depends on the smoothness of $F_\dagger$. Here, $F_\ast$ can be thought of as a `best approximation' of $F_\dagger$ in the class $\mW$.

The following lemma provides the necessary bound for $\jast(\lambda,R)$, crucially using entropy estimates for Besov spaces as well as our local Lipschitz assumptions on $\mathcal G$. [See Lemma 19 in \cite{NVW18} for an analogous argument in the Sobolev case.]



\begin{lem}\label{lem:psast}
Let $s=(\alpha+\kappa)/d$ and let $\mW=\tv$, for $\tv$ defined in \eqref{eq:vi}. Under the assumptions of Theorem~\ref{thm:plsfr}, there exists a  constant $\coni>0$ such that for all $\lambda, R>0$ and $F_\ast\in\tv$,
\begin{equation}\label{upperbound}
    \psast(\lambda,R)\coloneqq R+\coni\big(R^{1+\frac1{2s}}\lambda^{-\frac1{2 s}}(1+(R^2/\lambda)^{\frac{\gamma}{2s}})\big)
\end{equation}
is an upper bound for $\jast(\lambda,R)$.
\end{lem}
\begin{proof}
Let $\rho, \lambda, R>0$. We start by estimating the metric entropy appearing in $\jast(\lambda,R)$. 
Note that by the definition of $\tals$, we have
\begin{equation}\label{eq:vir}\vast(\lambda,R)\subseteq \tv(R^2/\lambda) \coloneqq\{ F\in \bk: \|F\|_{B^\alpha_{11}(\mO)} \le R^
2/\lambda \}.
\end{equation}
Furthermore, note that there exists $c_1>0$ such that for all $\phi\in\tv$, it holds
\begin{equation}\label{eq:hkbd}
    \norm{\phi}_{\hk}\leq c_1\norm{\phi}_{H^{-\kappa}(\mO)}.
\end{equation}
Indeed, let $\chi\in C^\infty_c(\mO)$ be a smooth cut-off function, with $\chi\equiv1$ on $\tilde{K}$, where $\tilde{K}$ is as in \eqref{eq:vi}. Then for any $\phi\in\tv$, the definition of the dual norms together with Lemma \ref{lem:chibd} in the Supplement and the definition of $\tilde{H}^\kappa(\mO)$ in \eqref{eq:htilde}, give
\begin{align*}\norm{\phi}_{\hk}=\sup_{\substack{\psi\in H^{\kappa}(\mO)\\\norm{\psi}_{H^\kappa(\mO)}\leq1}}|\int_\mO\psi\phi\chi| \lesssim \sup_{\substack{\tilde{\psi}\in \tilde{H}^{\kappa}(\mO)\\\norm{\tilde{\psi}}_{H^\kappa(\mO)}\leq1}}|\int_\mO\tilde{\psi}\phi|= \norm{\phi}_{(\tilde{H}^\kappa(\mO))^\ast}.\end{align*}
In particular, \eqref{eq:hkbd} follows once we recall that $H^{-\kappa}(\mO):=(\tilde{H}^\kappa(\mO))^\ast.$

Denote by $B^\alpha_{11}(\mO; r)$ the ball of radius $r>0$ in $\ba$. By Theorem 2 in Section 3.3.3 and Remark 1 in Section 1.3.1 of \cite{ET08}, and since by assumption $\alpha > d/2-\kappa$, we have
\[H\big(\rho,B^\alpha_{11}(\mO; r),\norm{\cdot}_{H^{-\kappa}(\mO)}\big)\lesssim (r/\rho)^{1/s},~~ \rho,r >0.\]


Now for some $\conk>0$ to be chosen below, let $m\coloneqq \conk\big(1+(R^2/\lambda)^{\gamma}\big)$. Then for $\conj$ from (\ref{eq:hkbd}) there exists some $\conl>0$ such that we can choose elements $F_1,\dots, F_N\in\vast(\lambda, R)$ with $N\leq \exp\big(\conl(\frac{R^{2}m}{\lambda\rho})^\frac1s\big),$
for which $\vast(\lambda, R)$ can be covered using the balls
\[\tilde{S}_i\coloneqq\big\{\psi\in\vast(\lambda,R):\norm{\psi-F_i}_{H^{-\kappa}(\mO)}\leq \frac{\rho}{m\conj}\big\}, \;i=1,\dots, N.\]

Now using \eqref{eq:hkbd}, assumption \eqref{eq:kgb} and the embedding of $\ba$ into $\cb$ for $\alpha>\beta+d$ (see the discussion after Assumption~\ref{ass:truth}), there exists $\conk>0$ large enough such that for all $i=1,\dots,N$ and for all $F\in\tilde{S}_i$,
\begin{equation}\label{eq:gbnd}
\begin{split}
\norm{\mG(F)-\mG(F_i)}_{\ltwod}&\leq m\norm{F- F_i}_{\hk}\le m\conj\norm{F-F_i}_{H^{-\kappa}(\mO)}\leq \rho,
\end{split}
\end{equation}
whence the balls 
\begin{equation}\label{eq:Gcover}S_i\coloneqq \{\psi\in\dast(\lambda,R): \norm{\psi-\mG(F_i)}_{\ltwod}\leq \rho\}, \;i=1,\dots,N,\end{equation}
form a covering of $\dast(\lambda,R)$. We thus obtain the bounds
\begin{align}\label{eq:entropybd}
H(\rho, \dast(\lambda,R),\norm{\cdot}_{\ltwod})&\lesssim (\frac{R^{2}m}{\lambda\rho})^\frac1s, \\
\int_0^{2R}H^\frac12(\rho,\dast(\lambda,R),\norm{\cdot}_{\ltwod})d\rho&\lesssim\int_0^{2R}(\frac{R^{2}m}{\lambda \rho})^{\frac1{2s}}d\rho\lesssim \lambda^{-\frac1{2 s}}R^{1+\frac1{2 s}}(1+(R^2/\lambda)^\frac\gamma{2 s}),\nonumber
\end{align}
where we used that $s>1/2$ since $\alpha>d/2-\kappa$. This shows that there exists $\coni>0$ such that for $\psast$ as defined in the statement, it holds that $\jast(\lambda,R)\leq\psast(\lambda,R)$ and the proof is complete.
\end{proof}

Part \textbf{ii)} of Theorem~\ref{thm:plsfr} now follows from combining Proposition~\ref{thm:genplsfr} with the preceding lemma.

\begin{proof}[Proof of Theorem~\ref{thm:plsfr}, part \textbf{ii)}]

We apply Proposition~\ref{thm:genplsfr} with $\mW=\tv$, where $\tv$ is defined in \eqref{eq:vi}.
Let $\Psi_*(\cdot,\cdot)$ be the upper bound given by Lemma~\ref{lem:psast}. First, using the assumption that $\alpha\ge  d/2+d\gamma-\kappa$, we see that for every $\lambda>0$, $R\mapsto \Psi_*(\lambda,R)/R^2$ is non-increasing. Second, we verify that for some large enough constant $\tilde c$, the requirement (\ref{eq:psast}) is fulfilled for any $0<\eps<1$ and with $\delta=\tilde c\delta_\eps$ and 
$\lambda=\tilde c^2\delta_\eps^2$ (where $\de$ was defined in \eqref{eq:deltaeps}). Indeed, with $c$ from (\ref{upperbound}), we may rephrase (\ref{eq:psast}) as the requirement
\[c_1\eps\big(\delta +2c\delta^{1+\frac{1}{2s}}\lambda^{-\frac{1}{2s}}\big)=c_1\eps\big(\tilde c\delta_\eps+2c\tilde c^{1-\frac 1{2s}} \delta_\eps^{1-\frac{1}{2s}}\big) \le \tilde c^2\delta_\eps^2,~~~~ s=(\alpha+\kappa)/d. \] 
By standard manipulations of the exponents and since the expression in the middle scales at most linearly in $\tilde c$, we see that the above is fulfilled for $\tilde c$ large enough.

Since the preceding considerations were independent of $F_*\in \tilde V$, we have thus verified the assumptions to apply Proposition~\ref{thm:genplsfr}, which immediately implies Theorem~\ref{thm:plsfr}.
\end{proof}

\subsection{Proof of Theorem \ref{thm:plsfr}, part \textbf{i)}} \label{sec:existproof}
This proof combines ideas from \cite[Lemma 4.1]{ABDH18} and \cite[Section 7]{NVW18}. We begin with the following lemma which will be used to show continuity of the objective $\mathcal J$.
\begin{lem}\label{lem:exist}
	Suppose $Y_\eps$, $\mG$ and $\alpha$ are as in the statement of Theorem \ref{thm:plsfr}. Let $\bk$ and $\bk(R)$ be as in \eqref{eq:vi} and \eqref{eq:vir}, and for some $\frac{d}2<\tilde{\alpha}<\alpha$ let $\dist$ be the distance induced by the $B^{\tilde{\alpha}}_{11}(\mO)$-norm. 
	Then there exists a version of the Gaussian white noise process $\W$ in $\ltwod$ such that for all $R>0$, the map 
	\[\Psi:(\bk(R), \dist)\to \R,~~F\mapsto \ip{Y_\eps}{\mG(F)}_{\ltwod},\]
	is almost surely uniformly continuous.
\end{lem}

The proof of this lemma, which we include for the reader's convenience, is similar to the proof of \cite[Lemma 34]{NVW18}. The main difference is that we consider Besov instead of Sobolev spaces here.
\begin{proof} 
	
	For any $\delta>0$ define the modulus of continuity 
	\[M_\delta\coloneqq \sup_{F,H\in \bk(R), \dist(F,H)\leq \delta}\left|\ip{Y_\eps}{\mG(F)-\mG(H)}_{\ltwod}\right|,\]
	which is a random variable. Moreover, define the event 
	\[A\coloneqq\{\omega\in\Omega\;\,\Big|\,\; M_\delta\stackrel{\delta\to0}{\longrightarrow}0\},\]
	where $\Omega$ is the sample space supporting the law $\rp\coloneqq P^1_\W$ of $\W$, as in Section \ref{sec:fm}. It suffices to show that $\rp(A)=1$ and since $M_\delta$ is non-increasing in $\delta$, it is hence sufficient to show that $\E[M_\delta]\stackrel{\delta\to0}{\longrightarrow}0$.
	
	To show this we apply Dudley's theorem, \cite[Theorem 2.3.7]{GN16}, to the Gaussian process 
	\[(\W(\psi)\;:\;\psi\in\dr), \quad\dr\coloneqq\{\mG(F)\;\,|\,\;F\in \bk(R)\}.\]
	For any $\delta>0$, define 
	\[R_\delta\coloneqq \sup_{F,H\in \bk(R), \dist(F,H)\leq \delta}\norm{\mG(F)-\mG(H)}_{\ltwod}.\]
	We have that $\mG$ is continuous as a mapping from $(\bk(R), \dist)$ to $\ltwod$. Indeed, this can be seen using that the norm inducing $\dist$ is stronger than the $\ltwo$-norm (hence also than the $\hk$-norm) since $\tilde{\alpha}>\frac{d}2$ (see \cite[Theorem 3.3.1]{T83}), the continuous embedding of $B^\alpha_{11}(\mO)$ into $C^\beta(\mO)$ (see discussion after Assumption \ref{ass:truth}) and the $(\kappa,\gamma,\beta)$-regularity of $\mG$. Since $(\bk(R),\dist)$ is a compact metric space (see Theorem 4.33 in \cite{T08}), the Heine-Cantor theorem implies that $\mG$ is in fact uniformly continuous, hence we have that $R_\delta\stackrel{\delta\to0}{\longrightarrow}0$.
	
	By the same argument as in the proof of Lemma \ref{lem:psast} (with $m\coloneqq c(1+R^\gamma)$), we can use the $(\kappa,\gamma,\beta)$-regularity of $\mG$ (in the sense of \eqref{eq:kgb}) to obtain that
	\[H(\rho,\dr,\norm{\cdot}_{\ltwod})\lesssim\left(\frac{Rm}{\rho}\right)^{\frac{d}{\alpha+\kappa}}, \quad\rho>0.\]
	
	Combining, we get that 
	\begin{align*}
		\E[M_\delta]&\lesssim \E\left[\sup_{F,H\in\bk(R), \dist(F,H)\leq\delta}\left|\ip{\W}{\mG(F)-\mG(H)}_{\ltwod}\right|\right]\\
		&\leq\E\left[\sup_{\psi,\phi\in\dr, \norm{\psi-\phi}_{\ltwod}\leq R_\delta}\left|\ip{\W}{\psi-\phi}_{\ltwod}\right|\right]\\
		&\lesssim\int_0^{R_\delta}\left(\frac{Rm}{\rho}\right)^{\frac{d}{2(\alpha+\kappa)}}d\rho,
	\end{align*}
	where for the first inequality we used triangle and Cauchy-Schwarz inequalities together with $R_\delta\to0$, and for the third inequality Dudley's theorem (in particular, assertion (2.42) in \cite{GN16}, with $X(t)$ replaced by $\ip{\W}{\psi}_{\ltwod}$ which is a sub-Gaussian process with respect to the $L^2$-metric). The right hand side vanishes as $\delta\to0$, since $R_\delta\to0$ and since $\alpha>\frac{d}2-\kappa$, hence the proof is  complete.
\end{proof}

\begin{proof}[Proof of Theorem \ref{thm:plsfr}, part (i)] {Fix $\lambda,\eps>0$ and let $Y_\eps$ be a fixed draw from $P^\eps_{F_\dagger}$. Denote by $\toma(r)$ the closed ball in $\toma$ of radius $r>0$.} \\
	
	{{\bf Step 1: localization.} By our assumption on $\alpha$, the upper bound $\psast(\lambda,R)$ in Lemma \ref{lem:psast} satisfies $\psast(\lambda, R)/R^2\stackrel{R\to\infty}{\longrightarrow}0$. Hence there exists $\delta>0$ such that for all $R\ge \delta$, we have that $R^2\geq \conf\eps\psast(\lambda,R),$ for $\conf$ as in \eqref{eq:psast}. Thus applying Proposition \ref{thm:genplsfr} (which does not assume the existence of a maximizer of $\mathcal J$), we obtain that the events
		\[A_n\coloneqq \Big\{ \mathcal J \;\;\text{has a maximizer}\;\; \hat{F}\notin \tv\cap \toma(2^n)\Big\},\] are such that $\rp(A_n)\stackrel{n\to\infty}{\longrightarrow}0$, hence choosing $n\in\N$ large enough, we have that the equality
		\[\sup_{F\in\tv\cap \toma(2^n)}\mathcal J(F)=\sup_{F\in\tv} \mathcal J(F),\]
		holds with probability as close to one as desired.}
	
	{{\bf Step 2: local existence.}
		By Step 1, it suffices to show the almost sure existence of a maximizer of $\mathcal J$ over $\tv\cap \toma(2^n)$, for any $n\in\N$. Fix $n\in\N$. Suppose $\{F_j\}_{j=1}^\infty\subset \tv\cap \toma(2^n)$ is a maximizing sequence of $\mathcal J$.} The sequence  $\{\norm{F_j}_{\toma}\}$ is bounded from above, hence the same holds for $\{\norm{F_j}_{\ba}\}$ (recall that $\norm{\cdot}_{\toma}$ and $\norm{\cdot}_{\ba}$ are equivalent for functions which are compactly supported in $\mO$). Fix $\tilde{\alpha}$ such that $\beta+d<\tilde{\alpha}<\alpha$. By the compactness of the embedding of $\ba$ into $B^{\tilde{\alpha}}_{11}(\mO)$ (see Theorem 4.33 in \cite{T08}), there exists a subsequence $\{F_{j'}\}_{j'=1}^\infty$ converging strongly in $B^{\tilde{\alpha}}_{11}(\mO)$ to some $\hat{F}\in B^{\tilde{\alpha}}_{11}(\mO)$. {We will show that $\hat{F}\in \tv\cap \toma(2^n)$} and that $\hat{F}$ is a maximizer of~$\mathcal J$. 
	
	First notice that since $\tilde{\alpha}>\beta+d$, the $B^{\tilde{\alpha}}_{11}(\mO)$-norm is stronger than the supremum norm (see for example the discussion after Assumption \ref{ass:truth} which shows that $B^{\tilde{\alpha}}_{11}(\mO)$ is continuously embedded in $C^\beta(\mO)$). Since for all $j$ it holds $\supp(F_j)\subset \kmap$, the strong convergence $F_{j'}\to\hat{F}$ in $B^{\tilde{\alpha}}_{11}(\mO)$ implies that $\supp(\hat{F})\subset \kmap$. Therefore, it follows that $F_{j'}\to\hat{F}$ also in $\tomta$, and moreover that in order to show that {$\hat{F}\in\tv\cap\toma(2^n)$ it suffices to show that $\hat{F}\in\toma(2^n)$.}
	
	Since $\toma$ can be identified with a weighted $\ell^1$ space of single-index sequences (the space of sequences $a_k$ such that $k^{\frac{\alpha}d-\frac12}a_k$ is absolutely summable with the natural weighted $\ell^1$ norm),  $\toma$ can be identified with the dual of a normed sequence space, $c_\alpha$ (which contains sequences $a_k$ such that $k^{\frac12-\frac{\alpha}d}a_k\to0$, and has norm $\norm{(k^{\frac12-\frac{\alpha}d}a_k)}_{\ell^\infty}$). This is straightforward to show analogously to the proof that $\ell^1$ is the dual of the space of null sequences $c_0$, see \cite[Example 3, Section III.2]{RS80}. Notice that $c_\alpha$ is separable, since its dual is separable. By Banach-Alaoglu theorem (see \cite[Theorem 8.5 and Definition 8.1(4)]{A16} for the precise version we use), there exists a further subsequence $F_{j''}$ converging in the weak* topology to some {$\bar{F}\in \toma(2^n)$}. Now due to the fact that $c_{\tilde{\alpha}}\subset c_\alpha$, we have that the weak* topology of $\toma$ is stronger than the weak* topology of $\tomta$, hence we have that $F_{j''}\to\bar{F}$ also in the weak* topology of $\tomta$. Combining with the fact that the strong convergence $F_{j'}\to\hat{F}$ in $\tomta$ implies also the weak* convergence, 
	the uniqueness of the weak*-limit implies that $\bar{F}=\hat{F}$. We hence have that {$\hat{F}\in \toma(2^n)$.} 
	
	We next show that $\hat{F}$ is almost surely a maximizer of $\mathcal J$. We claim that
	\begin{align*}
		\mathcal J(\hat{F})&=2\ip{Y_\eps}{\mG(\hat{F})}_{\ltwod}-\norm{\mG(\hat{F})}^2_{\ltwod}-\lambda\norm{\hat{F}}_{\toma}\\
		&\geq 2\lim_{j''\to\infty}\ip{Y_\eps}{\mG(F_{j''})}_{\ltwod}-\lim_{j''\to\infty}\norm{\mG(F_{j''})}_{\ltwod}-\lambda \liminf_{j''\to\infty}\norm{F_{j''}}_{\toma}\\
		&=\limsup_{j''\to\infty}{\mathcal J(F_{j''})}.
	\end{align*} 
	Indeed, for the first two limits we use Lemma \ref{lem:exist} and the continuity of $\mG$ as shown in the proof of the same lemma, both combined with the convergence $F_{j''}\to\hat{F}$ in the strong topology of $\bta$. For the limit supremum, we use the lower semicontinuity of the $\toma$-norm in the weak* topology of $\toma$, combined with the convergence $F_{j''}\to\hat{F}$ in the weak* topology of $\toma$. Finally, recalling that $F_j$ is a maximizing sequence of $\mathcal J$ completes the proof.
\end{proof}

\subsection{An auxiliary result for estimators over a Besov ball}

We conclude this section with a variation of Theorem~\ref{thm:plsfr} where penalized estimators are restricted to a Besov-ball. The result is a consequence of the same techniques employed in the preceding proofs; it holds under the (weaker) regularity Assumption~\ref{ass:kgb} and will be crucially used in the proof of Theorem~\ref{thm:gen-contraction}. With $\tv$ and $\mathcal J_{\lambda,\eps}$ defined as in (\ref{eq:vi}) and (\ref{Jdef}), let us define the subclasses 
\begin{equation}\label{eq:vb}
\tv_B:=\{ F\in \tv: \|F\|_{\ba}\le B \},\; B>0,
\end{equation}
and associated \textit{restricted} penalised least squares estimators $\hat F^{(B)}\in \arg\max_{F\in \tv_B}\mathcal J_{\lambda,\eps}$.

\begin{thm}\label{thm:plsrbddnorm}
Let $\alpha, \kappa, \beta \ge 0$ be integers such that $\alpha>\beta+d$
and assume the forward map $\mG:\vi\to \ltwod$ satisfies the local Lipschitz assumption (\ref{eq:local-lip}) for any $F_1,F_2\in \tv\cap \cb$. Suppose that data arise as $Y_\eps\sim P^\eps_{F_\dagger}$ for some fixed $F_\dagger\in\vi$ and let $\de$ be given by \eqref{eq:deltaeps}.
Then for every $B>0$, a.s. under $P^\eps_{F_\dagger}$ there exists a maximiser $\hat F^{(B)}\in \arg\max_{F\in \tv_B}\mathcal J_{\lambda,\eps}$, and there exist large enough constants $C,C',M>0$ such that with $\lambda=C\delta_\eps^2$, any $F_\ast \in \tv_B$ and any $0<\eps<1,~L\ge M$,
    \begin{align}
    P^\eps_{F_\dagger}\big(\tals(\hat F^{(B)},F_\dagger)\geq 2(\tals(F_\ast, F_\dagger)+L^2\de^2)~\text{for some}~\hat F^{(B)}\in &\arg\max_{F\in \tv_B}\mathcal J_{\lambda,\eps}\big)\leq C'\exp\Big(-\frac{L^2\delta_\eps^2}{C'\eps^2}\Big).\label{eq:concbdb}
    \end{align}
\end{thm}

\begin{proof}
The almost sure existence of a maximiser is proved analogously to part {\textbf i)} of Theorem \ref{thm:plsfr}; indeed the lower semicontinuity of the $\toma$-norm in the weak* topology of $\toma$, allows us to additionally establish that $\norm{\hat{F}}_{\toma}\leq B$, since any maximising sequence belongs to $\tv_B$.

Let $R>0$. We apply Proposition \ref{thm:genplsfr} with $\mW=\tv_B$, where the relevant entropy integral $\jast(\lambda,R)$ to bound now depends on $B$.
     In order to bound this integral, we notice that by the same line of argument from Lemma \ref{lem:psast}, for some constant $c=c_B>0$,
    \[ \Psi_*^{(B)}(\lambda,R):= R+c_B(R^{1+\frac{1}{2s}}\lambda^{-\frac{1}{2s}}) \] is an upper bound for $\jast(\lambda,R)$, now using the Lipschitz estimate (\ref{eq:local-lip}) in place of (\ref{eq:kgb}). It is moreover evident that $1+\frac{1}{2s}<2$, since by assumption we have $\alpha>d+\beta>d/2-\kappa$. It thus follows that $R\mapsto \Psi_*^{(B)}(\lambda,R)/R^2$ is non-increasing in $R>0$. Thus, like in the proof of Theorem \ref{thm:plsfr}, choosing $\lambda=\tilde c^2\delta_\eps^2$ and $\delta=\tilde c\delta_\eps$ for some $\tilde c>0$ large enough completes the proof.
\end{proof}


\section{Regularity and concentration of Besov priors}\label{sec:prior} In this section, we summarize a number of regularity and concentration results for $B^\alpha_{11}$-Besov prior distributions from Definition~\ref{def:prior}, which are needed throughout this paper. 
\begin{lem}[Support]\label{lem:supp}
Consider the probability measures $\tp$ and $\Pi$ from Definition~\ref{def:prior}.
\begin{enumerate}[label=\textbf{\roman*)}]    
\item\label{lem:bb} For any $\gamma\geq1$ it holds
   \[ \tp(B^b_{\gamma\gamma}(\rd))=
    \begin{cases}
	1, \;&\text{if}\;\;b<\alpha-d, \\ 
	0, \;&\text{if}\;\;b\geq\alpha-d. 
\end{cases}\]
\item\label{lem:cb} Assume $\alpha>1+d$ and let $b$ be an integer such that $1\le b<\alpha-d$. Then it  holds 
\[\Pi(\cbee)=\tp(\cbeer)=1.\]
\end{enumerate}
\end{lem}
\begin{proof}
Part \ref{lem:bb} follows in a similar way as Lemma 5.2 in \cite{ADH21} using ideas going back to Lemma 2 \cite{LSS09}. For the reader's convenience, we include the proof of the first leg which is the part used in our proofs.
Using the wavelet characterization of the Besov norm from display \eqref{eq:wavenorm} of the Supplement, the fact that the univariate Laplace random variables $\xi_{kl}$ in Definition \ref{def:prior} have finite polynomial moments and that $L_k\simeq 2^{dk}$, we have 
\begin{align*}\E_{\tp}[\norm{F}_{B^b_{\gamma\gamma}(\rd)}^\gamma]&=
 \sum_{k=1}^\infty2^{\gamma k(b+\frac{d}2-\frac{d}\gamma)}\sum_{l=1}^{L_k}2^{(\frac{d}2-\alpha)\gamma k}\E[|\xi_{kl}|^\gamma] \\
&\lesssim \sum_{k=1}^\infty2^{\gamma k(b+\frac{d}2-\frac{d}\gamma)}2^{dk}2^{(\frac{d}2-\alpha)\gamma k}
=\sum_{k=1}^\infty 2^{\gamma k (b+d-\alpha)},
\end{align*}
where the last sum is finite since $b<\alpha-d$.

For part \ref{lem:cb} we adapt the technique of \cite{DHS12} (see Remark 2.1 and discussion following Proposition 2.2 ibid) to our setting.
By Lemma \ref{lem:chibd} of the Supplement, we have 
\[\tp(\cbeer)\leq \Pi(\cbee),\] hence it suffices to show that $\tp(\cbeer)=1$. 

Given $b<\alpha-d$, choose $\gamma\ge1$ large enough so that $\frac{d}{\gamma}<\alpha-d-b$, and fix $b'$ such that $b+\frac{d}\gamma<b'<\alpha-d$. 

First, since $b'>b+\frac{d}\gamma$, the following embedding holds:
\begin{equation}\label{eq:claimedembd}B^{b'}_{\gamma\gamma}(\rd)\subset \cbeer.\end{equation}
Indeed, for $b''>b$ non-integer,  by \cite[Equation 2.5.7/9]{T83} we have
\[B^{b''}_{\infty\infty}(\rd)=C^{b''}(\rd)\subset C^b(\rd).\] Choosing $b''$ so that $b<b''<b'-\frac{d}\gamma$, by \cite[Equation 2.7.1/12]{T83}, we have  
\[B^{b'}_{\gamma\gamma}(\rd)\subset B^{b''}_{\infty\infty}(\rd),\]
hence the embedding \eqref{eq:claimedembd} is verified.

Then, since $b'<\alpha-d$, the proof can be completed using part \ref{lem:bb}.
\end{proof}


The Laplace wavelet prior $\tp$ is \emph{logarithmically concave} \cite[Lemma 3.4]{ABDH18}, which implies a Fernique-type theorem \cite{CB74}. This leads to the following basic concentration inequality.

\begin{lem}[$C^b$-concentration]\label{lem:fernique}
Consider $\tp$ as defined in Definition~\ref{def:prior}. Assume $\alpha>1+d$ and let $1\le b<\alpha-d$ be an integer. Then, there exist constants $\cona, \conb>0$ such that for any $r>0$
\[\tp(\norm{F}_{\cbeer}\geq r)\leq \cona e^{-\conb r}.\]
\end{lem}
\begin{proof}
Let $r>0$. By Lemma \ref{lem:supp} and \cite[Theorem 3.1]{CB74} (since $\tp$ is logarithmically concave \cite[Lemma 3.4]{ABDH18}), there exists $\conb>0$ sufficiently small such that 
\[\E_{\tp}[\exp(\conb\norm{F}_{\cbeer})]<\infty.\]
Combining with the (exponential) Markov inequality, we get that 
\begin{align*}
\tp(F: \norm{F}_{\cbeer}\geq r)\leq \E_{\tp}[\exp(\conb\norm{F}_{\cbeer})]e^{-\conb r}.
\end{align*}
\end{proof}

The next lower bound on small ball probabilities of $\tp$ is derived from Theorem 4.2 in \cite{A07}.
\begin{lem}[Small ball probabilities]\label{lem:sb}
Let $\tp$ be as in Definition~\ref{def:prior} and $\kappa>d-\alpha$. Then there exists a constant $C>0$ such that for any $0<\eta<1$,
\[-\log\tp(F:\norm{F}_{\hkr}\leq \del)\le C\del^{-\frac{d}{\kappa+\alpha-d}}.\]
\end{lem}
\begin{proof}
By the definition of $\tp$ and display \eqref{eq:wavesingle} of the Supplement, for some $c>0$ we have
\[\tp(F:\norm{F}_{\hkr}\le\del)\ge \rp\Big(\sum_{\ell=1}^\infty \ell^{\frac{d-2\alpha-2\kappa}d}\xi_\ell^2 \le c\del^2\Big),\]
where $\xi_\ell$ are independent and identically distributed standard Laplace real random variables. Thus, Theorem 4.2 in \cite{A07} (with $\mu=-\frac12+\frac{\alpha}d+\frac{\kappa}d$ and $p=2$) implies the desired result for any sufficiently small $\eta$. For $\eta<1$ bounded away from zero, the result is trivial since $\tp(F:\norm{F}_{\hkr}\lesssim\del)$ is then bounded away from zero.
\end{proof}


Associated to the measure $\tp$ from (\ref{eq:Ftilde}) are the spaces $\tsh$ and $\tom$ defined below, see \cite{ADH21}:
\begin{equation}\label{eq:tom}
\tom=\Big\{h\in\R^\infty: \sum_{k=1}^\infty2^{(\alpha-\frac{d}2)k}\sum_{l=1}^{L_k} |h_{kl}|<\infty\Big\}, \;\norm{h}_{\tom}=\sum_{k=1}^\infty2^{(\alpha-\frac{d}2)k}\sum_{l=1}^{L_k} |h_{kl}|,
\end{equation}
\begin{equation}\label{eq:tsh}\tsh=\Big\{h\in\R^\infty: \sum_{k=1}^\infty 2^{(2\alpha-d)k}\sum_{l=1}^{L_k}h_{kl}^2<\infty\Big\}, \;\norm{h}_{\tsh}=\Big(\sum_{k=1}^\infty2^{(2\alpha-d)k}\sum_{l=1}^{L_k} h_{kl}^2\Big)^\frac12.
\end{equation}
  On the one hand, the weighted-$\ell^1$ space $\tom$ determines the loss in probability under $\tp$ for non-centered balls compared to centered ones, see Lemma~\ref{lem:shsb} below. On the other hand, the weighted-$\ell^2$ space $\tsh$ is the space of admissible shifts of $\tp$, that is the space of shifts which give rise to equivalent measures to $\tp$. This space is relevant to our analysis, through the concentration inequality given in Lemma~\ref{lem:tal}. 
 
Notice that $\tom$ and $\tsh$ can be identified via the expansion \eqref{eq:wavexp} to subspaces of $L^2(\rd)$. In particular, for 
functions compactly supported on $\mO$, the norms of $\tom, \tsh$ are equivalent with the norms of the spaces $\ba, \had,$ respectively.%

\begin{lem}[Decentering]\label{lem:shsb}
Consider $\tp$ defined in Definition~\ref{def:prior}. Then for any $h\in \tom$ and any symmetric convex Borel-measurable $A\subset L^2(\rd)$, it holds 
\[\tp(h+A)\geq e^{-\norm{h}_{\tom}}\tp(A).\]
\end{lem}

\begin{proof}[Proof of Lemma \ref{lem:shsb}]
In \cite[Proposition 2.11]{ADH21}, Lemma \ref{lem:shsb} is proved for $A$ being a centered ball. An inspection of the proof shows that the result also holds for any symmetric convex Borel-measurable set $A$.
\end{proof}



By Lemma~\ref{lem:supp}, both the spaces $\bard$ and $\hard$ have zero measure under $\tp$. However, the bulk of the mass of $\tp$ concentrates on suitable, sufficiently large balls in these spaces, as long as they are `enlarged' by any set $A$ of positive measure. Specifically, we have the following concentration inequality, which is derived from Talagrand's work in \cite{T94}.


\begin{lem}[Two-level concentration]\label{lem:tal}
Consider $\tp$ defined in Definition~\ref{def:prior}. Then there exists a constant $\Lambda>0$, such that for any Borel-measurable $A\subset \hkr$ and any $r>0$, it holds that
\begin{align*}\tp\Big(F=F_1+F_2+F_3:~ F_1\in A, &~\norm{F_2}_{\bard}\leq r, ~\norm{F_3}_{\hard}\leq \sqrt{r},\\& F_i\in{\rm span}\{\subcol\}, i=1,2,3\Big)\geq 1-\frac1{\tp(A)}\exp(-r/\Lambda).\end{align*}
\end{lem}

\begin{proof}
Since sequences $h\in\tsh$ (resp. $\tom$) can be uniquely identified to functions $F_h\in\hard$ (resp. $\bard$), while there exist constants $c_{\tsh}, c_{\tom}>0$ such that $\norm{h}_{\tsh}\leq c_{\tsh}\norm{F_h}_{\hard}$ and $\norm{h}_{\tom}\leq c_{\tom}\norm{F_h}_{\tom}$, 
we obtain that there exists $c>0$, such that for any $r>0$,
\begin{align*}\tp\Big(F=F_1+F_2+F_3:~ F_1\in A, ~\norm{F_2}_{\bard}\leq r, &~\norm{F_3}_{\hard}\leq \sqrt{r},\\&~F_i\in{\rm span}\{\subcol\}, i=1,2,3\Big)\\\geq\tp\Big(F=F_1+F_2+F_3:~ F_1\in A, ~\norm{F_2}_{\tom}\leq cr, &~\norm{F_3}_{\tsh}\leq \sqrt{cr}, \\&~F_i\in{\rm span}\{\subcol\}, i=1,2,3\Big).\end{align*}
In the proof of \cite[Proposition 2.15]{ADH21}, it is shown that there exists a fixed $K'>0$, such that (in sequence-space) for any $\tp$-measurable set $A\subset \R^\infty$ we have
\[\tp\big(F=F_1+F_2+F_3:~ F_1\in A, ~\norm{F_2}_{\tom}\leq r, ~\norm{F_3}_{\tsh}\leq \sqrt{r}\big)\geq 1-\frac1{\tp(A)}\exp(-r/K').\] The result follows by taking $\Lambda=K'/c$.
\end{proof}

\begin{acks}[Acknowledgments]
The authors would like to thank Richard Nickl for suggesting this collaboration and for funding SA's research visit to Cambridge during which this work was initiated. We thank three anonymous referees, the AE and the editor for several helpful remarks. SW gratefully acknowledges financial support by the European Research Council, ERC grant agreement 647812 (UQMSI) as well as (during the finishing stages of this work) the Air Force Office of Scientific Research under award number FA9550-20-1-0397.
 \end{acks}
\bigskip

\begin{center}\large{\textbf SUPPLEMENTARY MATERIAL}\end{center}

\setcounter{section}{0}
\renewcommand\thesection{\Alph{section}}

\section{Proofs for Section \ref{sec:pderes}}\label{pde-proofs}
\begin{proof}[Proof of Theorem \ref{div-fwdrate}]
\textit{\textbf{i)}.} First, we note that composition with the link function $\Phi:\R\to (K_{min}, \infty)$ introduces a one-to-one correspondence between $f\in \mathcal F$ from (\ref{Fdef}) and the parameter
\[F:=\Phi^{-1} \circ f\in \{F\in B^\alpha_{11}:\supp(F)\subseteq K\}\]
without positivity constraint. Since moreover the prior $f\sim \Pi^f_\eps$ from (\ref{f-prior}) can be seen as the pushforward under $\Phi$ of the Laplace prior $F\sim \Pi_\eps$ from (\ref{prior}), we wish to apply Theorem \ref{thm:gen-contraction} with the forward map $\mathcal G(F)=u_{\Phi\circ F}$, which in turn requires the verification of Assumptions \ref{ass:truth} and \ref{ass:kgb}.

The PDE solution $\mathcal G(F)\in L^2(\mathcal O)$ is well defined whenever $F\in C^2_\chi(\mO)$, by virtue of classical Schauder theory \cite{GT98}.  Thus, by the hypotheses of Theorem \ref{div-fwdrate}, Assumption \ref{ass:truth} and \eqref{domain-lb-req} are verified with $\kappa=1$ and $\beta=2$. To verify the local Lipschitz condition (\ref{eq:local-lip}), fix $R>0$ and let $F,H\in C^2_\chi(\mathcal O)$ such that $\|F\|_{C^2(\mathcal O)}\vee \|H\|_{C^2(\mathcal O)}\le R$. We write $f=\Phi\circ F,~h=\Phi\circ H$. Then, by arguing as in the proof of Theorem 9 in \cite{NVW18} (noting that Lemmas 21 and 22 there as well as the derivations in that proof hold for any $f, h\in C^2$) and subsequently applying Lemma \ref{lem-boring} below, there are constants $a,a_R>0$ (the latter of which may depend on $R$) such that
\begin{align}
	\nonumber
	\|\mG(F)-\mG(H)\|_{L^2}&=\|u_f-u_h\|_{L^2}\leq a (1+\|f\|_{C^1}^2)(1+\|h\|_{C^1})\|h-f\|_{(H^1)^*}\nonumber\\
	\label{div-loc-lip}
	\le &  a_R (1+\|F\|_{C^2}^4\vee \|H\|_{C^2}^4)\|F-H\|_{(H^1)^*} \le a_R(1+R^4) \|F-H\|_{(H^1)^*}. 
\end{align}
Thus, (\ref{eq:local-lip}) holds with the choices $\kappa=1$, $\gamma=4$, $\beta=2$ and $C_R=a_R(1+R^4)$. Theorem \ref{thm:gen-contraction} yields part \textbf{i)} of the theorem. 


\textit{\textbf{ii)}.} Again by Theorem \ref{thm:gen-contraction}, as well as Lemma \ref{lem-boring} below, there are some constants $m,m'>0$ such that
\begin{equation}\label{eq:Cbconc}\Pi_\eps^f\big(f:\|f\|_{C^{\alpha-d-1}}>m'\big| Y_\eps\big)\le \Pi_\eps\big(F:\|F\|_{C^{\alpha-d-1}}>m \big| Y_\eps\big)= O_{P_{f_0}^\eps}(e^{-\de^2/\eps^2}).\end{equation}
Thus, we may restrict the remainder of the proof to the event (of sufficiently high posterior probability)
\begin{equation}\label{Adef}
	A=\big\{ f : \|f\|_{C^{\alpha-d-1}(\mO)} \le M, ~\|u_f - u_{f_0}\|_{\ltwo}\le M' \delta_\eps \big\}
\end{equation}
where $M, M'>0$ are sufficiently large constants. Since $\alpha>d+2$, we have in particular that $\norm{f}_{C^2}\le M$  for any $f\in A$.


Next, we claim that for $f\in A$ it holds $u_f\in H^{\alpha-d}(\mO)$ and that
\begin{equation}\label{sobolev-bd}
	\sup_{f\in A}\|u_{f}\|_{H^{\alpha-d}}\le M'',
\end{equation} 
for some constant $M''>0$ (independent of $\eps$). For $\alpha>3d/2+3$, this follows from Lemma 23 in \cite{NVW18} with $\alpha-d-1$ in place of $\alpha$, as well as the embedding $C^{\alpha-d-1}(\mO)\subseteq H^{\alpha-d-1}(\mO)$. For $\alpha\le 3d/2+3$, one argues analogously to the proof of Lemma 23 of \cite{NVW18}, replacing the Sobolev multiplication inequality used there by the multiplication inequality \[\norm{fg}_{H^b}\lesssim\norm{f}_{C^b}\norm{g}_{H^b}, \]
which holds for any integer $b\ge 1$. As in \cite{NVW18}, it then follows that
\begin{equation}\label{eq:lem23}
	\norm{u_f}_{H^{b+1}}\lesssim1+\norm{f}_{C^b(\mO)}^{b^2+b},
\end{equation} 
for any $f\in\mF\cap C^b(\mO)$, and our claim follows by choosing $b=\alpha-d-1$.


We are now in a position to apply the stability estimate from Lemma 24 of \cite{NVW18} with $f_1=f_0$ and $f_2=f\in A$. [Note that $f_0\in C^{\alpha-d-1}\subset C^1$ and $u_{f_0}\in C^2$ under the given assumptions, and that the hypothesis $\alpha>d/2+2$ in Lemma 24 of \cite{NVW18} is only needed to ensure existence of a classical solution, and can be replaced by $f\in C^2$ here.] This yields that
\begin{equation}\label{lem24}
	\|f-f_0\|_{\ltwo}\lesssim  \|f\|_{C^1(\mO)}\|u_{f}-u_{f_0}\|_{H^2(\mO)}\lesssim \|u_{f}-u_{f_0}\|_{H^2(\mO)}.
\end{equation}
Combining this with (\ref{sobolev-bd}) and a standard interpolation inequality for Sobolev spaces (see, e.g., \cite{LM72}), we obtain
\begin{equation}\label{stab-interpol}
	\begin{split}
		\|f-f_0\|_{\ltwo}&\lesssim 
		\|u_f-u_{f_0}\|_{\ltwo}^{\frac{\alpha-d-2}{\alpha-d}}\|u_f-u_{f_0}\|_{H^{\alpha-d}(\mO)}^{\frac{2}{\alpha-d}} \lesssim \|u_f-u_{f_0}\|_{\ltwo}^{\frac{\alpha-d-2}{\alpha-d}}\lesssim \delta_\eps ^{\frac{\alpha-d-2}{\alpha-d}}.
	\end{split}
\end{equation}
Thus, using the first part of the theorem and \eqref{eq:Cbconc}, it follows that for some constants $L,c>0$,
\begin{equation*}
	\begin{split}
		\Pi_\eps^f \big(f: \|f-f_0\|_{\ltwo}\ge L\de^{\frac{\alpha-d-2}{\alpha-d}}\big| Y_\eps \big)\le \Pi_\eps^f\big( A^c|Y_\eps\big)= O_{P^\eps_{f_0}}(e^{-c\delta_\eps^2/\eps^2}),
	\end{split}
\end{equation*}
which completes the proof.
\end{proof}

\begin{proof}[Proof of Theorem \ref{div-maprate}]
Since $\Phi$ is now additionally assumed to be a \textit{regular} link function, part \textbf{iii)} of Lemma D.2 below implies that the constant $a_R$ in (\ref{div-loc-lip}) can be chosen independently of $R>0$. In particular, (\ref{eq:kgb}) is satisfied for $\kappa=1,~\beta=2$ and $\gamma=4$. Since we also have $\alpha\ge 9d/2-1=d/2+d\gamma-\kappa$, parts \textbf{i)} and \textbf{ii)} now follow from Theorem \ref{thm:plsfr} (specifically, via Corollary \ref{cor-mse}) applied to the map $\mG:F\mapsto u_{\Phi\circ F}$. Indeed, this follows from the fact that for $F:=\Phi^{-1}\circ f$, $\mathcal G(F)=u_{\Phi\circ F}$ and $\mathcal J_{\lambda,\eps}$ from (\ref{Jdef}), we have $\mathcal J_{\lambda,\eps}(F)=\mathcal I(f)$ and thus also $\mathcal G(\Fpls)=u_{\fpls}$.

Part \textbf{iii)} follows from a similar argument as in the proof of Theorem 11 of \cite{NVW18}; we thus include here some details but leave the rest to the reader. For notational convenience, let us abbreviate $\hat f\equiv \fpls,~\hat F\equiv \Fpls\equiv \Phi^{-1}\circ \fpls$. First, note that the preceding application of Theorem \ref{thm:plsfr}, in combination with Lemma 7.9 below, the embedding (\ref{embedding}) and the norm equivalence $\|\cdot\|_{\ba}\simeq \|\cdot\|_{\toma}$, yields that for some $c,c',M>0$ and all $L\ge M$, we have the concentration bound
\begin{align}
	P_{f_0}^\eps \big(\|u_{\hat f}-u_{f_0}\|_{L^2}^2&+\lambda_\eps \|\hat f\|_{C^{\alpha-d-1}}\ge cL^2\de^2 \big)\nonumber\\
	&\le P_{f_0}^\eps \big(\|\mG(\hat F)-\mG(F_0)\|_{L^2}^2+\lambda_\eps \|\hat F\|_{\toma}\ge c'L^2\de^2 \big)\le C\exp\Big(-\frac{L^2\de^2}{C\eps^2}\Big), \label{map-slices}
\end{align}
where $C>0$ can be taken to be the constant from Theorem \ref{thm:plsfr}.

Using the stability estimate (\ref{lem24}), arguing similarly as in (\ref{stab-interpol}) (but keeping track of constants more explicitly) and subsequently applying (\ref{eq:lem23}), we obtain that
\begin{equation}\label{map-interpol}
	\begin{split}
		\|\hat f-f_0\|_{L^2} &\lesssim \|\hat f\|_{C^1}\|u_{\hat f}-u_{f_0}\|_{H^2} 
		\lesssim  \|\hat f\|_{C^1} \|u_{\hat f}-u_{f_0}\|_{L^2}^{\frac{\alpha-d-2}{\alpha-d}}\|u_{\hat f}-u_{f_0}\|_{H^{\alpha-d}}^{\frac{2}{\alpha-d}}\\
		&\lesssim \|u_{\hat f}-u_{f_0}\|_{L^2}^{\frac{\alpha-d-2}{\alpha-d}}\|\hat f\|_{C^1} \big(\|u_{\hat f}\|_{H^{\alpha-d}}  +\|u_{f_0}\|_{H^{\alpha-d}}\big)^{\frac{2}{\alpha-d}}\\
		&\lesssim \|u_{\hat f}-u_{f_0}\|_{L^2}^{\frac{\alpha-d-2}{\alpha-d}}\|\hat f\|_{C^1} \big(1+\|\hat f\|_{C^{\alpha-d-1}} +\|F_0\|_{B^\alpha_{11}}\big)^{\frac{2b^2+2b}{\alpha-d}},
	\end{split}
\end{equation}
where $b=\alpha-d-1$. Combining the preceding derivation with (\ref{map-slices}), akin to the proof of Theorem 11 in \cite{NVW18}, concludes this proof.
\end{proof}

\begin{proof}[Proof of Theorems \ref{thm:schr} and \ref{schr-maprate}]
Since the proofs of Theorems \ref{thm:schr} and \ref{schr-maprate} follow the same lines as those of Theorems \ref{div-fwdrate} and \ref{div-maprate}, we give a slightly briefer presentation here. We adopt the notation from those preceding proofs.
\smallskip

    To see part \textbf{i)} of Theorem \ref{thm:schr}, we verify Assumptions \ref{ass:truth} and \ref{ass:kgb} with $\kappa=2$, $\gamma=4$ and $\beta=2$ in order to apply Theorem \ref{thm:gen-contraction}, and as in the preceding proof we focus on verifying (\ref{eq:local-lip}). Fix any $R>0$. By arguing as in the proof of Theorem 12 of \cite{NVW18}, as well as using Lemma \ref{lem-boring} below, we see that there exist $a,a_R>0$ such that for any $F,H \in C^2_\chi(\mathcal O),~\|F\|_{C^2(\mathcal O)}\vee \|H\|_{C^2(\mathcal O)}\le R$,
    \begin{equation*}
        \begin{split}
        \|u_f-u_h\|_{L^2}&\le a(1+\|f\|_\infty)(1+\|h\|_\infty)\|f-h\|_{(H^2)^*}\\
        &\le a_R(1+\|F\|_{C^2}^4\vee \|H\|_{C^2}^4)\|F-H\|_{(H^2)^*},
        \end{split}
    \end{equation*}
    yielding part \textbf{i)}.

    For part \textbf{ii)}, by virtue of the concentration bounds \eqref{eq:forwardcontraction}- (\ref{eq:Cbconcentration}) and Lemma \ref{lem-boring} we may again focus on the posterior event (on which the posterior concentrates)
	\begin{equation*}
		A=\big\{ f: \|f\|_{C^{\alpha-d-1}(\mO)} \le M, ~\|u_f - u_{f_0}\|_{\ltwo}\le M' \delta_\eps \big\},
	\end{equation*}
    for sufficiently large $M,M'>0$. Next, we claim that
	\begin{equation}\label{schr-sob-bound}
	    \sup_{f\in A}\|u_{f}\|_{H^{\alpha-d+1}}\le M'',
	\end{equation}  for some constant $M''>0$. 
	Indeed, we see this by modifying Lemma 27 of \cite{NVW18} in the same way as we modified Lemma 23 in \cite{NVW18} to obtain (\ref{eq:lem23}), which yields that for any integer $b\geq 1$ and $f\in\mathcal{F}\cap C^b(\mO)$,
	\begin{equation}\label{eq:schr-highreg}
	    \norm{u_f}_{H^{b+2}}\lesssim 1+\norm{f}_{C^b(\mO)}^{b/2+1}.
	\end{equation}
    
    \smallskip 
    
    The stability estimate from Lemma 28 of \cite{NVW18} yields that for some $c_1,c_2>0$ and any $f_0,f\in C^2$, we have 
	\begin{equation}\label{lem28}\norm{f-f_0}_{\ltwo}\leq c_1\big(e^{c_2\norm{f}_\infty}\norm{u_{f}-u_{f_0}}_{H^2(\mO)}+\norm{u_{f_0}}_{C^2(\mO)}e^{c_2\norm{f\vee f_0}_\infty}\norm{u_f-u_{f_0}}_{\ltwo}\big).
	\end{equation}
    Noting that $f_0\in C^2$ and $u_{f_0}\in C^2$ under the given assumptions, there exists $C>0$ such that for any $f\in A$,
	\[ \|f-f_0\|_{L^2}\le C \|u_f-u_{f_0}\|_{H^2}. \]
	Now, an interpolation argument analogous to (\ref{stab-interpol}) concludes the proof:
	\[ \|f-f_0\|_{L^2}\lesssim \|u_f-u_{f_0}\|_{L^2}^{\frac{\alpha-d-1}{\alpha-d+1}}\lesssim \de^{\frac{\alpha-d-1}{\alpha-d+1}}, ~~~ f \in A.\]
	
	\smallskip 
	
	It remains to prove Theorem \ref{schr-maprate}. Parts $\textbf{i)}$ and $\textbf{ii)}$ are proved analogously as for Theorem \ref{div-maprate}, noting that since $\kappa=2$ here, the additional assumption on $\alpha$ is $\alpha\ge 9d/2-2$. Part \textbf{iii)} now follows from a similar argument as given in (\ref{map-slices})-(\ref{map-interpol}) to prove part \textbf{iii)} of Theorem \ref{div-maprate}, using the stability estimate (\ref{lem28}) and (\ref{eq:schr-highreg}) in place of (\ref{lem24}) and (\ref{eq:lem23}) respectively.

\end{proof}

\section{Some facts about Besov spaces}\label{supp:besov}
For $s\ge 0$, $p,q\in [1,\infty]$ and integer $d\ge 1$, we denote by $B^s_{pq}(\R^d)$ the usual Besov spaces on $\R^d$ (see Section 2.2.1 in \cite{ET08} for definitions). For any (open) domain $\mathcal O\subseteq \R^d$, let $L^2(\mathcal O)$ denote the square-integrable functions on $\mathcal O$. Further we denote by $B^s_{pq}(\mO)$ the space of restrictions $f=g|_{\mO}$ of elements $g\in B^s_{pq}(\R^d)$ to $\mathcal O$, equipped with the quotient norm
\begin{equation}\label{besovpain}
	\|f\|_{B^s_{pq}(\mO)}= \inf_{g: f=g|_{\mO}} \|g\|_{B^s_{pq}(\R^d)},
\end{equation}
see \cite{ET08}, p.57. For $p=q=2$, $B^s_{22}$ coincides with the usual Sobolev space $H^s(\mO)$ of $s$-times weakly differentiable functions with derivatives in $L^2(\mathcal O)$. [For non-integer $s$, these are defined by interpolation]. 

Recall that $\subcol=\{\psi_{kl}\}_{k\ge1, 1\le l\le L_k}$ is the sub-collection of wavelets in the orthonormal wavelet basis $\Psi$, whose supports intersect $\mO$. For $F$ of the form \eqref{eq:wavexp}, the $B^s_{pq}(\rd)$-norms can be characterized via the coefficients $F_{kl}$
\begin{equation}\label{eq:wavenorm}
\norm{F}_{B^s_{pq}(\rd)}=\Bigg(\sum_{k=1}^\infty 2^{qk(s+\frac{d}2-\frac{d}p)}\Big(\sum_{l=1}^{L_k}|F_{kl}|^p\Big)^{q/p}\Bigg)^{1/q}, \quad s\in\R, 1\leq p,q<\infty.
\end{equation}
For $p=\infty$ or $q=\infty$, we replace the $\ell_p$ or $\ell_q$-norm with $\ell_\infty$, respectively; see, e.g., p. 370 in \cite{GN16}. When $p=q$, using $L_k\simeq 2^{dk}$ together with the asymptotic equivalence between the sequences $\gamma_{kl}=2^{sk}, k\in\N, 1\leq l\leq 2^{dk}$ and $\gamma_\ell=\ell^{s/d}, \ell \in \N$, for any $s\in\R$, and enumerating $\{F_{kl}\}_{k\in \N, 1\leq l\leq L_k}$ using a single index as $\{F_{\ell}\}_{\ell=1}^\infty$, the characterization \eqref{eq:wavenorm} can be simplified to 
\begin{equation}\label{eq:wavesingle}
 \norm{F}_{B^s_{pp}(\rd)}\simeq \Bigg(\sum_{\ell=1}^\infty \ell^{\frac{ps}d+\frac{p}2-1}|F_\ell|^p\Bigg)^\frac1p,\quad s\in\R, 1\leq p<\infty,  
\end{equation}
with the obvious modification for $p=\infty$.

\smallskip

We have the following two basic lemmas.

\begin{lem}\label{lem:chibd}
Let $K$ be a compact subset of $\mO$ and $\chi\in C^\infty_c(\mO)$ be a smooth cut-off function with $\chi\equiv 1$ on $K$. Moreover, suppose that $s, b, \kappa >0$, $p\in\{1,2\}$, and let $\mW=\mO$ or $\rd$. Then there exists some constant $C<\infty$ such that
\begin{align}
\norm{\chi F}_{B^s_{pp}(\mO)}&\le C \norm{F}_{B^s_{pp}(\mW)} ~~\text{for all}~ F\in B^s_{pp}(\mW),\label{besov-mult}\\
\norm{\chi F}_{\cbee}&\le C \norm{F}_{C^b(\mathcal W)} ~~\text{for all}~ F\in C^b(\mathcal W),\label{holder-mult}\\
\norm{\chi F}_{\hk}&\le C \norm{F}_{(H^\kappa(\mathcal W))^*} ~~\text{for all}~ F\in L^2(\mathcal W).\label{neg-mult}
\end{align}
\end{lem}
\begin{proof}
The second statement follows directly from Remark 1 on p.143 in \cite{T78} (and the quotient characterization (\ref{besovpain})). The first statement, for non-integer $b$, follows from Remark 1 on p.143, and for integer $b$ is immediate from the product rule. Finally, (\ref{neg-mult}) is proved by a basic duality argument,
\[ \|\chi F\|_{\hk}= \sup_{\phi:\|\phi\|_{H^\kappa(\mathcal O)}\le 1}\Big|\int F (\chi\phi) \Big|\lesssim \norm{F}_{(H^\kappa(\mathcal W)^*}. \]

\end{proof}

\begin{lem}\label{lem-boring}
    Suppose that $K_{min}\in \R$ and that $\Phi:\R\to (K_{min},\infty)$ is a smooth, one-to-one map with $\Phi'(x)>0$, $x\in \R$.
    
    \smallskip
    
    \textbf{i)} For any integer $k\ge 1$ and $R>0$ there exists a constant $c_{k,R}>0$ such that for any $F\in C^k(\mathcal O)$ with $\|F\|_{C^k}\le R$, we have $\|\Phi \circ F\|_{C^k}\le c_{k,R}(1+\|F\|_{C^k}^k)$.
    
    \textbf{ii)} For any $R>0$ there exists $c_R>0$ such that for $\kappa\in \{1,2\}$ and any $F,H\in C^\kappa(\mathcal O)$ with $\|F\|_{C^\kappa},\|H\|_{C^\kappa}\le R$, we have
    \[ \|\Phi \circ F-\Phi \circ H\|_{(H^\kappa)^*}\le c_R\|F-H\|_{(H^\kappa)^*} (1+\|F\|_{C^\kappa}^\kappa\vee\|H\|_{C^\kappa}^\kappa).\]

    \textbf{iii)} If $\Phi$ is moreover a \emph{regular link function} in the sense of Definition \ref{def:reglink}, then $c_{k,R}$ and $c_R$ may be chosen independently of $R>0$.
\end{lem}

\begin{proof}
Parts \textbf{i)} and \textbf{ii)} follow from the same arguments as in the proof of parts 2 and 4 of Lemma 29 in \cite{NVW18}. Indeed, this follows from the fact that fixed balls $\{F\in C^k(\mathcal O):\|F\|_{C^k(\mathcal O)}\le R\}$ possess a uniform bound in $\|\cdot\|_\infty$-norm, whence it indeed suffices to take into account the local suprema $\sup_{x\in [-R,R]}\big(|\Phi(x)|+...+ |\Phi^{(k)}(x)|\big)<\infty$ in the above-mentioned proof. Part \textbf{iii)} of this lemma is in fact identical to parts 2 and 4 of Lemma 29 in \cite{NVW18}.
\end{proof}

Finally, the following approximation result is needed for the proof of Theorem \ref{thm:gen-contraction}, in particular in the proof of Lemma \ref{lem:sieve}.
\begin{lem}\label{lem:approx} 
Assume $\alpha>d-\kappa$ and recall the sub-collection of wavelets $\subcol$ from \eqref{eq:subcol}. Then there exists a constant $a>0$ such that for all $M>1$, any $\eps>0$ and any $F\in{\rm span}\{\subcol\}$ with $\norm{F}_{\hard}\leq M\frac{\de}{\eps}$, there is a decomposition $F=F_1+F_2$ with $F_1, F_2\in{\rm span}\{\subcol\}$ satisfying $\norm{F_1}_{\hkr}\leq \frac{\de^3}{\eps^2}$ and $\norm{F_2}_{\bard}\leq a M^2\der$.
\end{lem}
\begin{proof}
Let $F\in{\rm span}\{\subcol\}$ identified with its sequence of coefficients  $\{F_{kl}\}_{k\ge 1, 1\le l\le L_k}$ with respect to $\subcol$. Assume $\norm{F}_{\hard}\leq  M\frac{\de}{\eps}$ and consider the approximation of $F$ by the truncated sequence of coefficients $F_{1:\K}=\{F_{kl}\}_{1\le k \le \K, 1\le l \le L_k}$,  where $\K$ is the minimal truncation point such that 
$\norm{F-F_{1:\K}}_{\hkr}\leq \frac{\de^3}{\eps^2}$. Recalling $\hkr=H^{-\kappa}(\R^d)$ and using \eqref{eq:wavenorm}, we can bound 
\begin{align*}
\norm{F-F_{1:\K}}_{\hkr}^2&=\sum_{k>\K} 2^{-2\kappa k}\sum_{l=1}^{L_k}F_{kl}^2=\sum_{k>\K} 2^{(d-2\alpha-2\kappa)k}2^{(2\alpha-d)k}\sum_{l=1}^{L_k}F_{kl}^2\\
&\leq 2^{(d-2\alpha-2\kappa)\K}\norm{F}_{\hard}^2\leq 2^{(d-2\alpha-2\kappa)\K}M^2\der,
\end{align*}where for the first inequality we used the assumption $d<\kappa+\alpha$.
It thus follows that \\$\K\sim \frac2{2\alpha+2\kappa-d}\log_2(M\frac{\eps}{\delta_\eps^2})$ (notice that $\frac{\eps}{\delta_\eps^2}=\eps^{\frac{d-2\kappa-2\alpha}{2\kappa+2\alpha+d}}\to\infty$ as $\eps\to0$).

By \eqref{eq:wavenorm} and the Cauchy-Schwarz inequality, the approximation $F_{1:\K}$ satisfies 
\begin{align*}
\norm{F_{1:\K}}_{\bard}&=\sum_{k=1}^\K 2^{(\alpha-\frac{d}2)k}\sum_{l=1}^{L_k}|F_{kl}|\\
&\leq (\sum_{k=1}^{\K}L_k)^\frac12\Big(\sum_{k=1}^{\K} 2^{(2\alpha-d)k}\sum_{l=1}^{L_k}F_{kl}^2\Big)^{\frac12}\leq a 2^{\frac{d}2\K}\norm{F}_{\hard},
\end{align*}
for some constant $a>0$  (depending only on $\subcol$), and where we used that $L_k=\mathcal{O}(2^{dk})$.
Combining the assumed bound on $\norm{F}_{\hard}$ with the expressions for $\K$ and $\delta_\eps$ (see \eqref{eq:deltaeps}), we then have \[\norm{F_{1:\K}}_{\bard}\leq aM^{\frac{2\kappa+2\alpha}{2\kappa+2\alpha-d}}\der\leq aM^2\der,\]
where for the last bound we used the assumption $\alpha>d-\kappa$. 

The claim follows by taking $F_2=F_{1:\K}$ and $F_1=F-F_{1:\K}$.
\end{proof}

\bibliography{tvbib}
\bibliographystyle{imsart-number}
\end{document}